\documentclass[twoside,10pt]{article}

\usepackage{a4}                       
\usepackage[english]{babel}             
\usepackage[T1]{fontenc}              
\usepackage[latin1]{inputenc}         
\usepackage{amsmath,amsfonts,amssymb} 
\usepackage{graphics,graphicx}        
\usepackage{subfigure}                
\usepackage{epsfig}
\usepackage{ae,aecompl}
\usepackage{color}

\def\R{\mathbb{R}}
\def\P{\mathbb{P}}
\def\Z{\mathbb{Z}}
\def\N{\mathbb{N}}
\def\C{\mathbb{C}}

\def\k{\textit {\textbf k}}

\newcommand\dps{\displaystyle }

\newtheorem{thm}{Theorem}

\newtheorem{lem}{Lemma} 

\newtheorem{remark}{Remark} 

\def\qed{\relax
     \ifmmode
       ~\hfill\Box
     \else
        \unskip\nobreak ~\hfill$\Box$
      \fi \par}

\newtheorem{Proof}{Proof}

\newenvironment{proof}{\begin{Proof}\rm}{\qed\end{Proof}}

\begin{document}

\title{Numerical analysis of nonlinear eigenvalue problems}
\author{Eric Canc\`es\footnote{Universit\'e Paris-Est, CERMICS,
    Project-team Micmac, INRIA-Ecole des Ponts,  6 \& 8 avenue Blaise
    Pascal, 77455 Marne-la-Vall\'ee Cedex 2, France.}, Rachida
  Chakir\footnote{UPMC Univ Paris 06, UMR 7598 LJLL, Paris, F-75005 France ;
CNRS, UMR 7598 LJLL, Paris, F-75005 France} 
$\,$ and Yvon Maday$^\dag$\footnote{Division of Applied Mathematics, Brown
  University, Providence, RI, USA} } 
%
%
\maketitle

\begin{abstract}
We provide {\it a priori} error estimates for variational approximations
of the ground state energy, eigenvalue and eigenvector of nonlinear elliptic
eigenvalue problems of the form $-\mbox{div} (A\nabla u) + Vu + f(u^2) u =
\lambda u$, $\|u\|_{L^2}=1$. We focus in particular on the Fourier spectral
approximation (for periodic problems) and on the $\P_1$ and $\P_2$
finite-element discretizations. Denoting by $(u_\delta,\lambda_\delta)$
a variational approximation of the ground state eigenpair $(u,\lambda)$,
we are interested in the convergence rates of 
$\|u_\delta-u\|_{H^1}$, $\|u_\delta-u\|_{L^2}$,
$|\lambda_\delta-\lambda|$, and the ground state energy,  
when the discretization parameter $\delta$ goes to zero. We prove in
particular that
if $A$, $V$ and $f$ satisfy certain conditions,
$|\lambda_\delta-\lambda|$ goes to zero as
$\|u_\delta-u\|_{H^1}^2+\|u_\delta-u\|_{L^2}$. We also show that under
more restrictive assumptions on $A$, $V$ and $f$, 
$|\lambda_\delta-\lambda|$ converges to zero as
$\|u_\delta-u\|_{H^1}^2$, thus recovering a standard result for {\em
  linear} elliptic eigenvalue problems. For the latter analysis,
we make use of estimates of the error $u_\delta-u$ in negative Sobolev
norms. 
\end{abstract}

\section{Introduction}

Many mathematical models in science and engineering give rise to
nonlinear eigenvalue problems. Let us mention for instance the
calculation of the vibration modes of a mechanical structure
in the framework of nonlinear elasticity, the
Gross-Pitaevskii equation 
describing the steady states of Bose-Einstein
condensates~\cite{GrossPitaevskii}, or the 
Hartree-Fock and Kohn-Sham equations used to calculate ground state
electronic structures of molecular systems in quantum chemistry and
materials science~(see~\cite{handbook} for a mathematical
introduction). 

The numerical analysis of {\em linear} eigenvalue problems has been
thoroughly studied in the past decades (see e.g. \cite{linear}). On the other
hand, only a few results on {\em nonlinear} eigenvalue problems
have been published so far~\cite{Zhou1,Zhou2}.

In this article, we focus on a particular class of nonlinear eigenvalue
problems arising in the study of variational models of the form
\begin{equation} \label{eq:min_pb_u}
I = \inf \left\{ E(v), \; v \in X, \; \int_{\Omega} v^2 = 1 \right\}
\end{equation}
where 
\begin{equation*}
\left| \begin{array}{l} 
\Omega \mbox{ is a regular bounded domain or a rectangular brick of $\R^d$  and } X = H^1_0(\Omega) \\
\mbox{or} \\
\Omega \mbox{ is the unit cell of a periodic lattice ${\cal R}$ of $\R^d$ and }
X = H^1_\#(\Omega)
\end{array} \right.
\end{equation*}
with $d=1$, $2$ or $3$,
and where the energy functional $E$ is of the form 
$$
E(v) = \frac 12 a(v,v) + \frac 12 \int_\Omega F(v^2(x)) \, dx
$$
with
$$
a(u,v) = \int_\Omega (A \nabla u) \cdot \nabla v + 
\int_\Omega V uv.
$$
Recall that if $\Omega$ is the unit cell of a periodic lattice ${\cal
  R}$ of $\R^d$, then for all $s \in \R$ and $k \in \N$,
\begin{eqnarray*}
H^s_\#(\Omega) & = & \left\{ 
v|_\Omega, \; v \in H^s_{\rm loc}(\R^d) \; | \; v \; {\cal
  R}\mbox{-periodic} \right\}, \\
C^k_\#(\Omega) & = & \left\{ 
v|_\Omega, \; v \in C^k(\R^d) \; | \; v \; {\cal
  R}\mbox{-periodic} \right\} .
\end{eqnarray*}
We assume in addition that 
\begin{eqnarray} 
\!\!\!\!\!\!\!\!\!\!\!\!\!\! 
 &\bullet& A \in (L^\infty({\Omega}))^{d\times d} \mbox{ and }
 A(x)  \mbox{ is symmetric for almost all } x \in \Omega \label{eq:Hyp1} \\ 
\!\!\!\!\!\!\!\!\!\!\!\!\!\! && \exists \alpha > 0 \mbox{ s.t. } \xi^T A(x) \xi
\ge \alpha |\xi|^2 
\mbox{ for all } \xi \in \R^d \mbox{ and almost all } x \in \Omega
\label{eq:Hyp2}  \\ 
\!\!\!\!\!\!\!\!\!\!\!\!\!\!
&& \nonumber \\ 
\!\!\!\!\!\!\!\!\!\!\!\!\!\!
 &\bullet& V \in L^p(\Omega) \mbox{ for some } p > \max(1,d/2) 
\label{eq:Hyp3}  \\ 
\!\!\!\!\!\!\!\!\!\!\!\!\!\!
&& \nonumber \\
\!\!\!\!\!\!\!\!\!\!\!\!\!\!
&\bullet& F \in C^1([0,+\infty),\R) \cap C^2((0,\infty),\R) \mbox{ and }
F'' > 0 \mbox{ on } (0,+\infty)  \label{eq:Hyp6}  \\  
\!\!\!\!\!\!\!\!\!\!\!\!\!\! 
&& \exists 0 \le q < 2, \; \exists C \in \R_+  \mbox{ s.t. } \forall t
\ge 0, \;  
|F'(t)| \le C (1+t^q) \label{eq:Hyp7} \\  
\!\!\!\!\!\!\!\!\!\!\!\!\!\! 
&& F''(t)t \mbox{ remains bounded in the vicinity
  of } 0.  \label{eq:Hyp7p} 
\end{eqnarray}
To establish some of our results, we will also need to make the
additional assumption that there exists $1 < r \le 2$ and $0 \le s \le
5-r$ such that
\begin{eqnarray} 
\!\!\!\!\!\!\!\!\!\!\!\!\!\!
&& \forall R > 0, \; \exists C_R \in \R_+ \mbox{ s.t. }
\forall 0 < t_1 \le R, \; \forall t_2 \in \R, \; \nonumber \\
\!\!\!\!\!\!\!\!\!\!\!\!\!\!
& &
\quad \left| F'(t_2^2)t_2-F'(t_1^2)t_2-2F''(t_1^2)t_1^2(t_2-t_1)\right| 
\le C_R \left(1+|t_2|^{s} \right) |t_2-t_1|^r. \label{eq:Hyp8}
\end{eqnarray}

Note that for all $1 < m < 3$ and all $c > 0$, the function $F(t)=c t^m$
satisfies (\ref{eq:Hyp6})-(\ref{eq:Hyp7p}) and
(\ref{eq:Hyp8}), for some $1 < r \le 2$. It satisfies (\ref{eq:Hyp8})
with $r=2$ if $3/2 \le m < 3$. 
This allows us to handle the Thomas-Fermi kinetic energy
functional ($m=\frac 53$) as well as the repulsive interaction in
Bose-Einstein condensates ($m=2$). 

\begin{remark}
Assumption (\ref{eq:Hyp7}) is sharp for $d=3$, but is useless for
$d=1$ and can be replaced with the weaker assumption that there exist
$q < \infty$ and $C \in \R_+$ such that $|F'(t)| \le C(1+t^q)$ for all
$t \in \R_+$, for $d=2$. Likewise, the condition $0 \le s \le 5-r$ in assumption
(\ref{eq:Hyp8}) is sharp for $d=3$ but can be replaced with $0 \le s <
\infty$ if $d=1$ or $d=2$.
\end{remark}

In order to simplify the notation, we denote by  $f(t)=F'(t)$.

Making the change of variable $\rho=v^2$ and noticing that $a(|v|,|v|) =
a(v,v)$ for all $v \in X$, it is easy to check that
\begin{equation} \label{eq:min_pb_rho}
I = \inf \left\{ {\cal E}(\rho), \; \rho \ge 0, \; \sqrt\rho \in X, \;
  \int_{\Omega} \rho = 1 \right\},
\end{equation}
where
$$
{\cal E}(\rho) = \frac 12 a(\sqrt\rho,\sqrt\rho) + \frac 12 \int_\Omega
F(\rho). 
$$

We will see that under assumptions (\ref{eq:Hyp1})-(\ref{eq:Hyp7}),
(\ref{eq:min_pb_rho}) has a unique 
solution $\rho_0$ and (\ref{eq:min_pb_u}) has exactly two solutions:
$u = \sqrt{\rho_0}$ and $-u$. Moreover, $E$ is $C^1$ on
$X$ and for all $v \in X$, $E'(v) = A_vv$ where
$$
A_v = -\mbox{div}(A \nabla \cdot) + V + f(v^2).
$$
Note that $A_v$ defines a self-adjoint operator on $L^2(\Omega)$, with
form domain $X$. The function $u$
therefore is solution to the Euler equation 
\begin{equation} \label{eq:Euler}
\forall v \in X, \quad \langle A_uu-\lambda
u,v \rangle_{X',X} = 0 
\end{equation}
for some $\lambda \in \R$ (the Lagrange multiplier of the constraint
$\|u\|_{L^2}^2 = 1$) and equation~(\ref{eq:Euler}), complemented with the
constraint $\|u\|_{L^2} = 1$, takes the form of the nonlinear eigenvalue
problem
\begin{equation}\label{eq:10}
\left\{ \begin{array}{l}
A_u u = \lambda u \\
\| u\|_{L^2} = 1. \end{array} \right. 
\end{equation}
In addition, $u \in C^0(\overline{\Omega})$, $u > 0$ in $\Omega$ and
$\lambda$ is the ground state eigenvalue 
of the linear operator $A_u$. An important result is that $\lambda$ is a
{\em simple} eigenvalue of $A_u$. It is interesting to note that
$\lambda$ is also the ground state eigenvalue of the {\em nonlinear}
eigenvalue problem  
\begin{equation} \label{eq:nonlinear_eigenvalue_pb_0}
\left\{ \begin{array}{l}
\mbox{search } (\mu,v) \in \R \times X \mbox{ such that} \\
A_v v = \mu v \\
\| v \|_{L^2} = 1, \end{array} \right. 
\end{equation}
in the following sense: if $(\mu,v)$ is solution to
(\ref{eq:nonlinear_eigenvalue_pb_0}) then either $\mu > \lambda$ or
$\mu=\lambda$ and $v= \pm u$. 
All these properties, except maybe the last one, are classical. For the
sake of completeness, their proofs are however given in the Appendix. 

Let us now turn to the main topic of this article, namely the derivation
of a priori error estimates for variational approximations of the ground
state eigenpair $(\lambda,u)$. We denote by $(X_\delta)_{\delta > 0}$ a
family of finite-dimensional subspaces of $X$ such that
\begin{equation} \label{eq:densite}
\min \left\{ \|u-v_\delta\|_{H^1}, \; v_\delta \in X_\delta \right\} \; 
\mathop{\longrightarrow}_{\delta \to 0^+} \; 0
\end{equation}
and consider the variational
approximation of (\ref{eq:min_pb_u}) consisting in solving
\begin{equation} \label{eq:min_pb_u_delta}
I_\delta = \inf \left\{ E(v_\delta), \; v_\delta \in X_\delta, \;
  \int_{\Omega} v_\delta^2 = 1 \right\}. 
\end{equation}
Problem (\ref{eq:min_pb_u_delta}) has at least one minimizer $u_\delta$,
which satisfies
\begin{equation}\label{eq:14}
\forall v_\delta \in X_\delta, \quad \langle A_{u_\delta}u_\delta-\lambda_\delta
u_\delta,v_\delta \rangle_{X',X} = 0 
\end{equation}
for some $\lambda_\delta \in \R$. Obviously, $-u_\delta$ also is a
minimizer associated with the same eigenvalue $\lambda_\delta$. On the
other hand, it is not known whether $u_\delta$ and $-u_\delta$
are the only 
minimizers of (\ref{eq:min_pb_u_delta}). One of the reasons why the
argument used in the infinite-dimensional setting cannot be transposed
to the discrete case is that the set
$$
\left\{\rho \; | \; \exists u_\delta  \in X_\delta  \mbox{ s.t. }
  \|u_\delta \|_{L^2} = 1, \; \rho = u_\delta ^2 \right\}
$$
is not convex in general. We will see however (cf. Theorem~\ref{Th:basic}) 
that for any family $(u_\delta)_{\delta > 0}$ of global
minimizers of (\ref{eq:min_pb_u_delta}) such that 
$(u,u_\delta) \ge 0$ for all $\delta > 0$, the following holds true
$$
\|u_\delta-u\|_{H^1} \mathop{\longrightarrow}_{\delta \to 0^+} \; 0.
$$
In addition, a simple calculation leads to
\begin{equation} \label{eq:estimate_eigenvalue}
\lambda_\delta-\lambda = 
\langle (A_u-\lambda)(u_\delta-u),(u_\delta-u) \rangle_{X',X} 
+ \int_\Omega w_{u,u_\delta} (u_\delta-u)
\end{equation}
where
$$
w_{u,u_\delta} =  u_\delta^2
\frac{f(u_\delta^2)-f(u^2)}{u_\delta-u}.
$$
The first term of the right-hand side of (\ref{eq:estimate_eigenvalue})
is nonnegative and goes to zero as $\|u_\delta-u\|_{H^1}^2$. 
We will prove in Theorem~\ref{Th:basic} that the second term
goes to zero at least as $\|u_\delta-u\|_{L^{6/(5-2q)}}$. Therefore,
$|\lambda_\delta-\lambda|$ converges to zero with $\delta$ at least as
$\|u_\delta-u\|_{H^1}^2+\|u_\delta-u\|_{L^{6/(5-2q)}}$.

\medskip

The purpose of this article is to provide more precise {\it a priori} error
bounds on $|\lambda_\delta-\lambda|$, as well as on
$\|u_\delta-u\|_{H^1}$, $\|u_\delta-u\|_{L^2}$ and
$E(u_\delta)-E(u)$. In Section~\ref{sec:basic}, we prove a 
series of estimates valid in the general framework described
above. 
We then turn to more specific examples, where the analysis can be pushed
further. 
In Section~\ref{sec:Fourier}, we concentrate on the discretization of
problem (\ref{eq:min_pb_u}) with
\begin{eqnarray}
&& \Omega=(0,2\pi)^d, \nonumber \\
&& X=H^1_\#(0,2\pi)^d, 
\nonumber \\
&& E(v) = \frac 12 \int_\Omega |\nabla v|^2 + \frac 12 \int_\Omega V v^2
+ \frac {1}{2} \int_\Omega F(v^2), 
\nonumber
\end{eqnarray}
in Fourier modes.
In Section~\ref{sec:FE}, we deal with the ${\mathbb P}_1$ and ${\mathbb
  P}_2$ finite element discretizations of problem (\ref{eq:min_pb_u}) with
\begin{eqnarray}
&& \Omega \mbox{ rectangular brick of $\R^d$}, \nonumber \\
&& X=H^1_0(\Omega),
\nonumber \\
&&  E(v) = \frac 12 \int_\Omega |\nabla v|^2 + \frac 12 \int_\Omega V v^2
+ \frac {1}{2} \int_\Omega F(v^2).
\nonumber
\end{eqnarray}

Lastly, we discuss the issue of numerical integration in
Section~\ref{sec:integration}.

\section{Basic error analysis}
\label{sec:basic}

The aim of this section is to establish error bounds on
$\|u_\delta-u\|_{H^1}$, $\|u_\delta-u\|_{L^2}$,
$|\lambda_\delta-\lambda|$ and $E(u_\delta)-E(u)$, in a general
framework. In the whole section, 
we make the assumptions (\ref{eq:Hyp1})-(\ref{eq:Hyp7p}) and
(\ref{eq:densite}), and we denote by $u$ the unique positive solution of
(\ref{eq:min_pb_u}) and by $u_\delta$ a minimizer of the discretized problem
(\ref{eq:min_pb_u_delta}) such that $(u_\delta,u)_{L^2} \ge 0$. We also
introduce the bilinear form $E''(u)$ defined on $X \times X$ by 
$$
\langle E''(u)v,w \rangle_{X',X} = \langle A_u v,w \rangle_{X',X} +
2 \, \int_\Omega f'(u^2)u^2 vw.
$$
When $F \in C^2([0,+\infty),\R)$, then $E$ is twice differentiable on
$X$ and $E''(u)$ is the second derivative of $E$ at $u$.

\medskip

\begin{lem} \label{lem:technical} 
There exists $\beta > 0$ and $M \in \R_+$ such that for all $v \in X$, 
\begin{eqnarray}
0 & \le &  \langle (A_u-\lambda) v,v \rangle_{X',X}  \le  M
\|v\|_{H^1}^2 \label{eq:Au-lambda_2} \\  
\beta \|v\|_{H^1}^2 & \le & \langle (E''(u)-\lambda) v,v \rangle_{X',X}  \le  M \|v\|_{H^1}^2.\label{eq:Euu-lambda}
\end{eqnarray}
There exists $\gamma > 0$ such that for all $\delta > 0$,
\begin{equation}
\gamma \|u_\delta-u\|_{H^1}^2 \le \langle (A_u-\lambda)
(u_\delta-u),(u_\delta-u) \rangle_{X',X} .  
 \label{eq:Au-lambda_1} 
\end{equation}
\end{lem}

\begin{proof} We have for all $v \in X$,
$$
\langle (A_u-\lambda) v,v \rangle_{X',X} \le \|A\|_{L^\infty}
\|\nabla v\|_{L^2}^2 + \|V\|_{L^p} \|v\|_{L^{2p'}}^2 + \|f(u^2)\|_{L^\infty}
\|v\|_{L^2}^2
$$
where $p'=(1-p^{-1})^{-1}$ and
$$
\langle (E''(u)-\lambda) v,v \rangle_{X',X} \le 
\langle (A_u-\lambda) v,v \rangle_{X',X} + 2 \|f'(u^2)u^2\|_{L^\infty}
\|v\|_{L^2}^2. 
$$
Hence the upper bounds in (\ref{eq:Au-lambda_2}) and (\ref{eq:Euu-lambda}). 
We now use the fact that $\lambda$, the lowest
eigenvalue of $A_u$, is simple (see Lemma~\ref{lem:theory} in the
Appendix). This 
implies that there exists $\eta > 0$ such that
\begin{equation} \label{eq:gap}
\forall v \in X, \quad 
\langle (A_u-\lambda) v,v \rangle_{X',X} \ge 
\eta (\|v\|_{L^2}^2-|(u,v)_{L^2}|^2) \ge 0.
\end{equation}
This provides on the one hand the lower bound
(\ref{eq:Au-lambda_2}), and leads on the other hand to the inequality
$$
\forall v \in X, \quad 
\langle (E''(u)-\lambda) v,v \rangle_{X',X} \ge 
2\int_{\Omega} f'(u^2)u^2 v^2.
$$
As $f' =F'' > 0$ in $(0,+\infty)$ and $u > 0$ in $\Omega$, we therefore have
$$
\forall v \in X \setminus \left\{0\right\}, \quad 
 \langle (E''(u)-\lambda) v,v \rangle_{X',X} > 0.
$$
Reasoning by contradiction, we deduce from the above inequality and the
first inequality in (\ref{eq:gap}) that there exists $\widetilde \eta >
0$ such that
\begin{equation} \label{eq:borne_inf_1}
\forall v \in X, \quad 
\langle (E''(u)-\lambda) v,v \rangle_{X',X} \ge \widetilde \eta \|v\|_{L^2}^2.
\end{equation}
Besides, there exists a constant $C \in \R_+$ such that
\begin{equation} \label{eq:borne_inf_2}
 \forall v \in X, \quad 
\langle (A_u-\lambda) v,v \rangle_{X',X} \ge \frac{\alpha}2 \|\nabla
v\|_{L^2}^2 - C \|v\|_{L^2}^2.
\end{equation}
Let us establish this inequality for $d=3$ (the case when $d=1$ is
straightforward and the case when $d=2$ can be dealt with in the same
way). For all $x \in X$, 
\begin{eqnarray*}
\langle (A_u-\lambda) v,v \rangle_{X',X} 
& = &
\int_\Omega (A\nabla v) \cdot \nabla v + \int_\Omega (V+f(v^2)-\lambda)
v^2  \nonumber \\ 
& \ge &  
\alpha \|\nabla v\|_{L^2}^2 - \|V\|_{L^p} \|v\|_{L^{2p'}}^2 + (f(0)-\lambda)
\|v\|_{L^2}^2 \nonumber \\
& \ge &  
\alpha \|\nabla v\|_{L^2}^2 - \|V\|_{L^p} \|v\|_{L^2}^{2-3/p}
\|v\|_{L^6}^{3/p} + (f(0)-\lambda) \|v\|_{L^2}^2 \nonumber \\
& \ge &  
\alpha \|\nabla v\|_{L^2}^2 - C_6^{3/p} \|V\|_{L^p} \|v\|_{L^2}^{2-3/p}
\|v\|_{H^1}^{3/p} + (f(0)-\lambda) \|v\|_{L^2}^2 \nonumber \\
& \ge & 
\frac \alpha 2 \|\nabla v\|_{L^2}^2 +
\left(f(0) - \lambda - \frac{3-2p}{2p}
\left( \frac{3C_6^2\|V\|_{L^p}^{2p/3}}{p\alpha} \right)^{3/(2p-3)} -
\frac \alpha 2 \right)
\|v\|_{L^2}^2, 
\end{eqnarray*}
where $C_6$ is the Sobolev constant such that $\forall v \in X$,
$\|v\|_{L^6} \le C_6 \|v\|_{H^1}$. The coercivity of
$E''(u)-\lambda$ (i.e. the lower bound in (\ref{eq:Euu-lambda})) is a
straightforward consequence of (\ref{eq:borne_inf_1}) and 
(\ref{eq:borne_inf_2}).

To prove (\ref{eq:Au-lambda_1}), we notice that
$$
\|u_\delta\|_{L^2}^2-|(u,u_\delta)_{L^2}|^2 \ge 1-(u,u_\delta)_{L^2}=\frac
12 \|u_\delta-u\|_{L^2}^2. 
$$
It therefore readily follows from (\ref{eq:gap}) that 
$$
\langle (A_u-\lambda) (u_\delta-u),(u_\delta-u) \rangle_{X',X} \ge
\frac{\eta}2 \|u_\delta-u\|_{L^2}^2.
$$
Combining with (\ref{eq:borne_inf_2}), we finally obtain
(\ref{eq:Au-lambda_1}). 
\end{proof}

\medskip

For $w \in X'$, we denote by $\psi_w$ the unique solution to the adjoint
problem 
\begin{equation} \label{eq:adjoint}
\left\{ \begin{array}{l}
\mbox{find } \psi_w \in u^\perp \mbox{ such that} \\
\forall v \in u^\perp, \quad \langle (E''(u)-\lambda) \psi_w,v
\rangle_{X',X} = \langle w,v \rangle_{X',X}, \end{array} \right.
\end{equation}
where
$$
u^\perp = \left\{ v \in X \; | \; \int_\Omega uv = 0 \right\}.
$$
The existence and uniqueness of the solution to (\ref{eq:adjoint}) is a
straightforward consequence of (\ref{eq:Euu-lambda}) and the Lax-Milgram
lemma. Besides,
\begin{equation} \label{eq:bound_psi}
\forall w \in L^2(\Omega), \quad \|\psi_w\|_{H^1} \le \beta^{-1}M \|w\|_{X'}
\le \beta^{-1}M \|w\|_{L^2}.
\end{equation}

\medskip

We can now state the main result of this section.

\medskip

\begin{thm} \label{Th:basic}
Under assumptions (\ref{eq:Hyp1})-(\ref{eq:Hyp7}) and
(\ref{eq:densite}), it holds
$$
\|u_\delta-u\|_{H^1} \mathop{\longrightarrow}_{\delta \to 0^+} \; 0.
$$
If in addition, (\ref{eq:Hyp7p}) is satisfied, then there exists $C \in
\R_+$ such that for all $\delta > 0$, 
\begin{equation} \label{eq:error_energy}
\frac \gamma 2 \|u_\delta-u\|_{H^1}^2 \le E(u_\delta)-E(u) \le 
\frac M 2 \|u_\delta-u\|_{H^1}^2 + C \|u_\delta-u\|_{L^{6/(5-2q)}},
\end{equation}
and
\begin{equation} 
|\lambda_\delta-\lambda|  \le  C \left( \|u_\delta-u\|_{H^1}^2 
+ \|u_\delta-u\|_{L^{6/(5-2q)}} \right). \label{eq:error_lambda}
\end{equation}
Besides, if assumption (\ref{eq:Hyp8}) is satisfied for some $1 < r \le
2$ and $0 \le s \le 5-r$, then 
there exists $\delta_0 > 0$ and $C \in \R_+$ such that for all $0 <
\delta < \delta_0$, 
\begin{eqnarray}
\|u_\delta-u\|_{H^1} & \le & C \min_{v_\delta \in X_\delta}
\|v_\delta-u\|_{H^1} \label{eq:error_H1} \\
\|u_\delta-u\|_{L^2}^2 & \le &   C \, \bigg( 
\|u_\delta-u\|_{L^2} \|u_\delta-u\|_{L^{6r/(5-s)}}^r \nonumber \\ & & 
\qquad 
+ \|u_\delta-u\|_{H^1} \min_{\psi_\delta \in X_\delta}
\|\psi_{u_\delta-u}-\psi_\delta\|_{H^1} \bigg). 
\label{eq:error_L2} 
\end{eqnarray}
Lastly, if $F''$ is bounded in the vicinity of $0$, there exists $C \in
\R_+$ such that for all $\delta > 0$,
\begin{equation} \label{eq:error_energy_2}
\frac \gamma 2 \|u_\delta-u\|_{H^1}^2 \le E(u_\delta)-E(u) \le 
C \|u_\delta-u\|_{H^1}^2.
\end{equation}
\end{thm}

\medskip

\begin{remark}
If $0 \le r+s \le 3$, then 
$$
\|u_\delta-u\|_{L^{6r/(5-s)}}^r \le  \|u_\delta-u\|_{L^2}^{(5-r-s)/2}
\|u_\delta-u\|_{L^6}^{(3r-5+s)/2}  
\le \|u_\delta-u\|_{L^2} \|u_\delta-u\|_{H^1}^{r-1}, 
$$
so that (\ref{eq:error_L2}) implies the simpler inequality
\begin{equation} \label{eq:error_L2_simpler}
\|u_\delta-u\|_{L^2}^2 \le C \|u_\delta-u\|_{H^1} \min_{\psi_\delta \in
  X_\delta}
\|\psi_{u_\delta-u}-\psi_\delta\|_{H^1}.
\end{equation}
\end{remark}

\medskip

\begin{proof}{\bf $\!\!\!\!\!$ of Theorem~\ref{Th:basic}}
We have
\begin{eqnarray}
\!\!\!\!\!\!\!\!\!\!\!\!\!\!\!\!\!\!
E(u_\delta ) - E(u) & = & \frac 12 \langle A_u u_\delta ,u_\delta  \rangle_{X',X} 
- \frac 12 \langle A_u u,u \rangle_{X',X} \nonumber \\
& & + \frac 12 \int_\Omega 
F\left(  u_\delta ^2 \right)
- F\left( u^2 \right)
- f\left(  u^2 \right) (u_\delta ^2-u^2) \nonumber \\
& = & \frac 12 \langle (A_u-\lambda) (u_\delta-u) ,(u_\delta-u)  \rangle_{X',X} 
\nonumber \\
& & + \frac 12 \int_\Omega 
F\left(  u_\delta \right)
- F\left( u^2 \right)
- f\left(  u^2 \right) (u_\delta ^2-u^2). \label{eq:Eudelta-Eu} 
\end{eqnarray}
Using (\ref{eq:Au-lambda_1}) and the convexity of $F$, we  
get 
$$
E(u_\delta ) - E(u) \ge  \frac \gamma 2 \|u_\delta -u\|_{H^1}^2.  
$$
Let $\Pi_\delta u \in X_\delta$ be such that
$$
\|u-\Pi_\delta u\|_{H^1} = \min \left\{ \|u-v_\delta\|_{H^1}, v_\delta
  \in X_\delta \right\}.  
$$
We deduce from (\ref{eq:densite}) that $(\Pi_\delta u)_{\delta > 0}$
converges to $u$ in $X$ when $\delta$ goes to zero. Denoting by
$\widetilde u_\delta  = \|\Pi_\delta u\|_{L^2}^{-1} \Pi_\delta u$ (which is
well defined, 
at least for $\delta$ small enough), we also have  
$$
\lim_{\delta \to 0^+} \|\widetilde u_\delta - u \|_{H^1} = 0.
$$
The functional $E$ being strongly continuous on $X$, we obtain
$$
 \|u_\delta -u\|_{H^1}^2 \le \frac 2\gamma \left(E(u_\delta ) -
   E(u)\right) \le  \frac 2\gamma \left(E(\widetilde
u_\delta ) - E(u)\right) \mathop{\longrightarrow}_{\delta \to 0^+} 0.
$$
It follows that there exists $\delta_1 > 0$ such that
$$
\forall 0 < \delta \le  \delta_1, \quad \|u_\delta\|_{H^1} \le 2
\|u\|_{H^1}, \quad \|u_\delta-u\|_{H^1} \le \frac 12.
$$
We then easily deduce from (\ref{eq:Eudelta-Eu}) the upper bounds in
(\ref{eq:error_energy}) and (\ref{eq:error_energy_2}). 

Next, we remark that
 \begin{eqnarray}
\lambda_\delta-\lambda & = &\langle E'(u_\delta ),u_\delta
\rangle_{X',X} - \langle E'(u),u \rangle_{X',X} \nonumber  \\ 
&=& a (  u_\delta, u_\delta) - a(u,u) + \int_\Omega
f(u_\delta^2)u_\delta^2 - \int_\Omega f(u^2)u^2 \nonumber \\ 
& =& a(u_\delta-u ,u_\delta-u) + 2 a(u, u_\delta-u) + \int_\Omega
 f(u_\delta^2)u_\delta^2 - \int_\Omega f(u^2)u^2 \nonumber \\
 &=&a (u_\delta-u ,u_\delta-u) + 2\lambda \int_\Omega   u  (u_\delta-u) 
- 2\int_\Omega f(u^2)u(u_\delta-u)  \nonumber \\ & & \qquad + \int_\Omega
 f(u_\delta^2)u_\delta^2 - \int_\Omega f(u^2)u^2 \nonumber \\
 &=& a (u_\delta-u ,u_\delta-u) -  \lambda \|u_\delta-u\|_{L^2}^2
- 2\int_\Omega f(u^2)u(u_\delta-u) \nonumber \\ & & \qquad + \int_\Omega
 f(u_\delta^2)u_\delta^2 - \int_\Omega f(u^2)u^2 \nonumber \\  
  &= & \langle (A_u-\lambda)(u_\delta-u ),(u_\delta-u ) \rangle_{X',X}
+   \int_\Omega w_{u,u_\delta } (u_\delta-u )
\label{Cal_lambda}
\end{eqnarray}
where
$$
w_{u,u_\delta} = u_\delta^2
\frac{f(u_\delta^2)-f(u^2)}{u_\delta-u}.
$$
As $u \in L^\infty(\Omega)$, we have
$$
|w_{u,u_\delta}| \le \left| \begin{array}{lll}
\dps 12 u \sup_{t \in (0,4\|u\|_{L^\infty}^2]} F''(t)t & \quad & \mbox{if }
|u_\delta| < 2 u \\
\dps 2 \left( |f(u_\delta^2)|+\max_{t \in [0,\|u\|_{L^\infty}^2]} |f(t)|
\right) |u_\delta| & \quad & \mbox{if }
|u_\delta| \ge 2 u,
\end{array} \right.
$$
and we deduce from assumptions (\ref{eq:Hyp7})-(\ref{eq:Hyp7p}) that
$$
|w_{u,u_\delta}| \le C (1+|u_\delta|^{2q+1}),
$$
for some constant $C$ independent of $\delta$.
Using (\ref{eq:Au-lambda_2}), we
therefore obtain that for all $0 < \delta \le \delta_1$,
\begin{eqnarray}
|\lambda_\delta-\lambda|   & \le  &  M \| u_\delta-u \|_{H^1}^2 + 
\| w_{u,u_\delta}\|_{L^{6/(2q+1)}}  \|u_\delta-u\|_{L^{6/(5-2q)}} \nonumber \\
& \le  &  M \| u_\delta-u \|_{H^1}^2 + C
(1+\|u_\delta\|_{H^1}^{2q+1}) \|u_\delta-u\|_{L^{6/(5-2q)}} \nonumber \\
& \le & C \left( \| u_\delta-u \|_{H^1}^2 + \|u_\delta-u\|_{L^{6/(5-2q)}} \right),
\label{eq:estim_36}
\end{eqnarray}
where $C$ denotes constants independent of $\delta$.

In order to evaluate the $H^1$-norm of the error $u_\delta-u$, we first
notice that 
\begin{equation}\label{eq:4:1:H1}
\forall v_\delta \in X_\delta, \quad 
\| u_\delta-u \|_{H^1} \le \| u_\delta - v_\delta  \|_{H^1} 
+\| v_\delta  - u \|_{H^1},
\end{equation}
and that 
 \begin{eqnarray}
 \| u_\delta  - v_\delta \|_{H^1}^2 & \le & \beta^{-1} \,  
\langle (E''(u) - \lambda)(u_\delta  - v_\delta ),
(u_\delta  - v_\delta ) \rangle_{X',X} \nonumber \\
& = & \beta^{-1}  \bigg( \langle (E''(u) - \lambda)(u_\delta  - u),
(u_\delta  - v_\delta ) \rangle_{X',X} \nonumber \\ 
&  & \quad \qquad +\langle (E''(u) - \lambda)(u- v_\delta ),(u_\delta  -
  v_\delta ) \rangle_{X',X} \bigg). 
 \label{eq:4:3:H1}
 \end{eqnarray}
For all $w_\delta \in X_\delta$ 
\begin{eqnarray}
\!\!\!\!\!\!\!\!\!\!\!\!\!\!\!\!\!\!\!\!\!
& & \!\!\!\!\!\!
\langle( E''(u) - \lambda)(u_\delta-u),w_\delta \rangle_{X',X} 
\nonumber \\ 
\!\!\!\!\!\!\!\!\!\!\!\!\!\!\!\!\!\!\!\!\!
& &
= - \int_\Omega \left( f(u^2_\delta) u_\delta -f(u^2)u_\delta - 2
    f'(u^2)u^2(u_\delta-u) \right) w_\delta  +(\lambda_\delta-
  \lambda)\int_\Omega u_\delta w_\delta. \label{eq:new36}
\end{eqnarray}
On the other hand, we have for all $v_\delta \in X_\delta$ such that
$\|v_\delta\|_{L^2}=1$, 
$$
\int_\Omega u_\delta (u_\delta-v_\delta) = 1 - \int_\Omega  u_\delta
v_\delta
= \frac 12 \|u_\delta-v_\delta\|_{L^2}^2.
$$
Using (\ref{eq:Hyp8}) and (\ref{eq:estim_36}), we therefore obtain that
 for all $0 < \delta \le \delta_1$ and all $v_\delta \in X_\delta$ such
 that $\|v_\delta\|_{L^2}=1$,
 \begin{eqnarray}
\left| \langle (E''(u) - \lambda)(u_\delta-u),(u_\delta-v_\delta)
  \rangle_{X',X}\right|
& \le & C \bigg( \|u_\delta-u\|_{L^{6r/(5-s)}}^r \| u_\delta-v_\delta
\|_{H^1} 
\nonumber \\ & &
\!\!\!\!\!\!\!\!\!\!\!\!\!\!\!\!\!\!\!\!\!\!\!\!\!\!\!\!\!\!\!\!\!\!\!\!\!\!\!\!\!\!\!\!\!\!\!\!\!\!\!\!\!\!\!\!\!\!\!\!   
+ \left( \|u_\delta-u\|_{H^1}^2 + \|u_\delta-u\|_{L^{6/(5-2q)}} \right) 
\| u_\delta-v_\delta \|_{L^2}^2  \bigg). \label{eq:4:3:H2}
 \end{eqnarray}
It then follows from (\ref{eq:Euu-lambda}), (\ref{eq:4:3:H1}) and
(\ref{eq:4:3:H2})  that for all $0 < \delta \le \delta_1$ and all
$v_\delta \in X_\delta$ such that $\|v_\delta\|_{L^2}=1$,  
\begin{eqnarray*}
 \| u_\delta  - v_\delta \|_{H^1} & \le &  C 
 \left(   \|u_\delta-u\|_{H^1}^r + \|u_\delta-u\|_{H^1} \| u_\delta  -
   v_\delta \|_{H^1}  + \|v_\delta-u \|_{H^1} \right) .
\end{eqnarray*}
Combining with (\ref{eq:4:1:H1}) we obtain that there exists $0 <
\delta_2 \le 
\delta_1$ and $C \in \R_+$ such that for all $0 < \delta \le \delta_2$ and all
$v_\delta \in X_\delta$ such that $\|v_\delta\|_{L^2}=1$, 
$$
\|u_\delta -u \|_{H^1} \le C  \|v_\delta-u \|_{H^1}.   
$$
Hence, for all $0 < \delta \le \delta_2$
$$
\|u_\delta -u \|_{H^1} \le C J_\delta \qquad \mbox{where} \qquad 
J_\delta = \min_{v_\delta \in X_\delta \, | \, 
\|v_\delta\|_{L^2}=1} \|v_\delta-u \|_{H^1}.   
$$
We now denote by
$$
\widetilde J_\delta = \min_{v_\delta \in X_\delta} \|v_\delta-u \|_{H^1},
$$
and by $u_\delta^0$ a minimizer of the above minimization problem. We
know from (\ref{eq:densite}) that $u_\delta^0$ converges to $u$ in $H^1$
when $\delta$ goes to zero. Besides,
\begin{eqnarray*}
J_\delta & \le & \|u_\delta^0/\|u_\delta^0\|_{L^2} - u \|_{H^1} \\
& \le & \|u_\delta^0 - u \|_{H^1} +
\frac{\|u_\delta^0\|_{H^1}}{\|u_\delta^0\|_{L^2}} 
\left| 1- \|u_\delta^0\|_{L^2} \right| \\
& \le & \|u_\delta^0 - u \|_{H^1} +
\frac{\|u_\delta^0\|_{H^1}}{\|u_\delta^0\|_{L^2}} 
\|u-u_\delta^0\|_{L^2} \\
& \le & \left( 1 +  \frac{\|u_\delta^0\|_{H^1}}{\|u_\delta^0\|_{L^2}}
\right) \widetilde J_\delta.
\end{eqnarray*} 
For $0 < \delta \le \delta_2 \le \delta_1$, we have
$\|u_\delta^0-u\|_{H^1} \le \|u_\delta-u\|_{H^1} \le 1/2$, and therefore
$\|u_\delta^0\|_{H^1} \le \|u\|_{H^1}+1/2$ and $\|u_\delta^0\|_{L^2} \ge
1/2$, yielding $J_\delta \le 2(\|u\|_{H^1}+1) \widetilde J_\delta$.
Thus (\ref{eq:error_H1}) is proved.

Let $u_\delta ^*$ be the orthogonal projection, for the $L^2$ inner product, of 
$u_\delta$ on the affine space $\left\{v \in L^2(\Omega) \, | \, \int_\Omega uv=1
\right\}$. One has 
$$
u_\delta ^* \in X, \qquad u_\delta ^*-u \in u^\perp,  \qquad 
u_\delta^*-u_\delta  = \frac 12 \|u_\delta-u \|_{L^2}^2 u, 
$$
from which we infer that
\begin{eqnarray*}
\|u_\delta -u\|_{L^2}^2 & = & \int_\Omega (u_\delta-u )(u_\delta ^*-u) + 
\int_\Omega (u_\delta-u)(u_\delta - u_\delta ^*) \\
& = &  \int_\Omega (u_\delta-u)(u_\delta^*-u) - \frac 12
\|u_\delta-u\|_{L^2}^2 \int_\Omega (u_\delta-u) u  \\
& = &  \int_\Omega (u_\delta-u)(u_\delta^*-u) + \frac 12
\|u_\delta-u\|_{L^2}^2 \left( 1 - \int_\Omega u_\delta u \right) \\
& = &  \int_\Omega (u_\delta-u)(u_\delta^*-u) + \frac 14 \|u_\delta-u\|_{L^2}^4 \\
& = &   \langle u_\delta-u  ,u_\delta^*-u \rangle_{X',X}  + \frac 14
\|u_\delta-u \|_{L^2}^4 \\ 
& = &   \langle (E''(u)-\lambda) \psi_{u_\delta-u}, u_\delta^*-u
\rangle_{X',X} + \frac 14 \|u_\delta-u \|_{L^2}^4 \\
& = &   \langle  (E''(u)-\lambda)  (u_\delta-u), \psi_{u_\delta-u}
\rangle_{X',X} \\ 
& & + \frac 12 \|u_\delta-u\|_{L^2}^2 
 \langle  (E''(u)-\lambda) u , \psi_{u_\delta-u} \rangle_{X',X} 
+ \frac 14 \|u_\delta-u \|_{L^2}^4 \\ 
& = &   \langle  (E''(u)-\lambda)  (u_\delta-u), \psi_{u_\delta-u}
\rangle_{X',X} \\ 
& & +  \|u_\delta-u\|_{L^2}^2 \int_\Omega f'(u^2) u^3 \psi_{u_\delta-u} 
+ \frac 14 \|u_\delta-u \|_{L^2}^4 .
\end{eqnarray*}
For all $\psi_\delta \in X_\delta$, it therefore holds
\begin{eqnarray*}
\|u_\delta-u\|_{L^2}^2 & = &   
\langle  (E''(u)-\lambda)  (u_\delta-u),\psi_\delta
\rangle_{X',X} \\
& & + \langle  (E''(u)-\lambda)(u_\delta-u),\psi_{u_\delta-u}-\psi_\delta
\rangle_{X',X} \\ 
& &  + \|u_\delta-u\|_{L^2}^2 \int_\Omega f'(u^2) u^3 \psi_{u_\delta-u} 
+ \frac 14 \|u_\delta-u \|_{L^2}^4 . 
\end{eqnarray*}
From (\ref{eq:new36}), we obtain that for all $\psi_\delta \in X_\delta
\cap u^\perp$,
\begin{eqnarray*}
\!\!\!\!\!\!\!\!\!\!\!\!\!\!\!\!\!\!\!\!\!
& & \!\!\!\!\!\!
\langle( E''(u) - \lambda)(u_\delta-u),\psi_\delta \rangle_{X',X} 
 \\ 
\!\!\!\!\!\!\!\!\!\!\!\!\!\!\!\!\!\!\!\!\!
& &
= - \int_\Omega \left( f(u^2_\delta) u_\delta -f(u^2)u_\delta - 2
    f'(u^2)u^2(u_\delta-u) \right) \psi_\delta  +(\lambda_\delta-
  \lambda)\int_\Omega (u_\delta-u) \psi_\delta
\end{eqnarray*}
and therefore that for all $\psi_\delta \in X_\delta
\cap u^\perp$,
 \begin{eqnarray}
\left| \langle (E''(u) - \lambda)(u_\delta-u),\psi_\delta
  \rangle_{X',X}\right|
& \le & C \bigg( \|u_\delta-u\|_{L^{6r/(5-s)}}^r 
\nonumber \\ & &
\!\!\!\!\!\!\!\!\!\!\!\!\!\!\!\!\!\!\!\!\!\!\!\!\!\!\!\!\!\!\!\!\!\!\!\!\!\!\!\!\!\!\!\!\!\!\!\!\!\!\!\!\!\!\!\!\!\!\!\!   
+ \| u_\delta-u \|_{L^{6/5}} 
\left( \|u_\delta-u\|_{H^1}^2 + \|u_\delta-u\|_{L^{6/(5-2q)}} \right) 
 \bigg) \|\psi_\delta\|_{H^1} . \label{eq:new37}
 \end{eqnarray}

Let $\psi_\delta^0 \in X_\delta \cap u^\perp$ be such that
$$
\|\psi_{u_\delta-u}-\psi_\delta^0\|_{H^1} = \min_{\psi_\delta \in
  X_\delta  \cap u^\perp} \|\psi_{u_\delta-u}-\psi_\delta\|_{H^1}.
$$  
Noticing that $\|\psi_\delta^0\|_{H^1} \le  \|\psi_{u_\delta-u}\|_{H^1}
\le \beta^{-1} M\|u_\delta-u\|_{L^2}$, we obtain from 
(\ref{eq:Euu-lambda}) and (\ref{eq:new37}) that
there exists $C \in 
\R_+$ such that for all $0 < \delta \le \delta_1$,
\begin{eqnarray*}
\|u_\delta-u\|_{L^2}^2 &\le & C \bigg( 
  \|u_\delta-u\|_{L^2} \\ & & \!\!\!\!\!\!\!\!\!\!\!\!\!\!\!\!\!
\times  \left( \|u_\delta-u\|_{L^{6r/(5-s)}}^r +
\|u_\delta-u\|_{L^{6/5}}
\left( \|u_\delta-u\|_{H^1}^2 + \|u_\delta-u\|_{L^{6/(5-2q)}} \right) \right)
\\
& &   + \|u_\delta-u\|_{H^1}
\|\psi_{u_\delta-u}-\psi_\delta^0\|_{H^1} 
+ \|u_\delta-u\|_{L^2}^3 + 
\|u_\delta-u\|_{L^2}^4 \bigg).
\end{eqnarray*}
Therefore, there exists $0 < \delta_0 \le \delta_2$ and $C \in \R_+$
such that for all $0 < \delta \le \delta_0$,
$$
\|u_\delta-u\|_{L^2}^2 \le  C \, \bigg(
  \|u_\delta-u\|_{L^2} \|u_\delta-u\|_{L^{6r/(5-s)}}^r
+   \|u_\delta-u\|_{H^1}
\|\psi_{u_\delta-u}-\psi_\delta^0\|_{H^1} \bigg). 
$$
Lastly, denoting by $\Pi_{X_\delta}^0$ the orthogonal projector on
$X_\delta$ for the $L^2$ inner product, a simple calculation leads to
\begin{equation} \label{eq:minuperp}
\forall v \in u^\perp, \quad \min_{v_\delta \in X_\delta  \cap u^\perp}
\|v_\delta - v \|_{H^1}
\le \left( 1 + \frac{\|\Pi_{X_\delta}^0 u \|_{H^1}} 
  {\|\Pi_{X_\delta}^0 u \|_{L^2}^2} \right) 
\min_{v_\delta \in X_\delta} \|v_\delta - v \|_{H^1},
\end{equation}
which completes the proof of Theorem~\ref{Th:basic}.
\end{proof}

\begin{remark} \label{rem:negative_Sobolev}
In the proof of Theorem~\ref{Th:basic}, we have obtained bounds on
$|\lambda_\delta-\lambda|$ from (\ref{Cal_lambda}), using $L^p$
estimates on $w_{u,u_\delta}$ and $(u_\delta-u)$ to control the second
term of the right hand side. Remarking that 
\begin{eqnarray*}
\nabla w_{u,u_\delta} &=& -u \,
\frac{f(u^2)u-f(u_\delta^2)u-2f'(u_\delta^2)u_\delta^2(u-u_\delta)}{(u_\delta-u)^2}
\, \nabla u_\delta \\ 
&& -u_\delta \,
\frac{f(u_\delta^2)u_\delta-f(u^2)u_\delta-2f'(u^2)u^2(u_\delta-u)}{(u_\delta-u)^2}
\, \nabla u \\ 
&& + 2uu_\delta \left(f'(u_\delta^2) \, \nabla u_\delta+f'(u^2)\nabla
  u\right)
+ 2u_\delta \,\frac{f(u_\delta^2)-f(u^2)}{u_\delta-u} \,  \nabla u_\delta \, ,
\end{eqnarray*}
we can see that if $u_\delta$ is
uniformly bounded in $L^\infty(\Omega)$ and if $F$ satisfies (\ref{eq:Hyp8})
for $r=2$ and is such that $F''(t)t^{1/2}$ is bounded in the vicinity of
$0$, then $w_{u,u_\delta}$ is
uniformly bounded in $X$. It then follows from (\ref{Cal_lambda}) that
$$
|\lambda_\delta-\lambda| \le C \left(\|u_\delta-u\|_{H^1}^2 +
  \|u_\delta-u\|_{X'} \right),
$$
an estimate which is an improvement of (\ref{eq:error_lambda}). In the
next two sections, we will see that this approach (or analogous strategies
making use of negative Sobolev norms of higher orders),
can be used in certain cases to obtain optimal estimates on
$|\lambda_\delta-\lambda|$ of the form  
$$
|\lambda_\delta-\lambda| \le C \|u_\delta-u\|_{H^1}^2,
$$
similar to what is obtained for the linear eigenvalue problem $-\Delta u
+ V u = \lambda u$.
\end{remark}

\section{Fourier expansion}
\label{sec:Fourier}

In this section, we consider the problem
\begin{equation} \label{eq:Fourier}
\inf \left\{ E(v), \; v \in X, \; \int_\Omega v^2=1 \right\},
\end{equation}
where
\begin{eqnarray}
&& \Omega=(0,2\pi)^d, \quad \mbox{with $d=1$, $2$ or $3$,} \nonumber \\
&& X=H^1_\#(\Omega), 
\nonumber \\
&& E(v) = \frac 12 \int_\Omega |\nabla v|^2 + \frac 12 \int_\Omega V v^2
+ \frac {1}{2} \int_\Omega F(v^2).
\nonumber
\end{eqnarray}
We assume that $V \in H^\sigma_\#(\Omega)$ for some $\sigma > d/2$ and
that the function $F$ satisfies (\ref{eq:Hyp6})-(\ref{eq:Hyp7p}),
(\ref{eq:Hyp8}) for some $1 < r \le 2$ and $0 \le s \le 5-r$, and is 
in $C^{[\sigma]+1,\sigma-[\sigma]+\epsilon}((0,+\infty),\R)$ (with the convention
that $C^{k,0}=C^k$ if $k \in \N$). 

The positive solution $u$ to (\ref{eq:Fourier}), which satisfies the
elliptic equation
$$
-\Delta u + V u + f(u^2) u = \lambda u,
$$
then is in $H^{\sigma+2}_\#(\Omega)$ and is bounded away from $0$. To
obtain this result, we have used the fact \cite{Sickel} that if $s >
d/2$, $g \in 
C^{[s],s-[s]+\epsilon}(\R,\R)$ and $v \in H^s_\#(\Omega)$, then $g(v) \in
H^s_\#(\Omega)$. 

A natural discretization of (\ref{eq:Fourier}) consists in using a
Fourier 
basis. Denoting by $e_k(x) = (2\pi)^{-d/2} e^{ik \cdot x}$, we have
for all $v \in L^2(\Omega)$,   
$$
v(x) = \sum_{\k\in \Z^d} \widehat v_k e_k(x), 
$$
where $\widehat v_k$ is the $k^{\rm th}$ Fourier coefficient of $v$:
$$
\widehat v_k := \int_\Omega v(x) \, \overline{e_k(x)} \, dx
= (2\pi)^{-d/2} \int_\Omega v(x) \, e^{-ik \cdot x} \, dx.
$$
The approximation of the solution to (\ref{eq:Fourier}) by the spectral
Fourier approximation is based on the choice 
$$
X_\delta = \widetilde X_N= \mbox{Span}\{ e_k, \; |k|_*\le N \},
$$ 
where $|k|_*$ denotes either the $l^2$-norm or the $l^\infty$-norm of $k$
(i.e. either $|k|=(\sum_{i=1}^d|k_i|^2)^{1/2}$ or $|k|_\infty=\max_{1 \le i \le
  d}|k_i|$). 
For convenience, the discretization parameter for this approximation
will be denoted as $N$. 

Endowing $H^r_\#(\Omega)$ with the norm defined by
$$
\|v\|_{H^r} = \left( \sum_{k \in \Z^d} \left(1+|k|_*^2\right)^r |\widehat v_k|^2
\right)^{1/2}, 
$$
we obtain that for all $s \in \R$, and all $v \in H^s_\#(\Omega)$, the
best approximation of $v$ in $H^r_\#(\Omega)$ for any $r \le s$ is
$$
\Pi_N v = \sum_{\k\in \Z^d, |\k|_* \le N} \widehat v_k e_k.
$$
The more regular $v$ (the regularity being measured in terms of the
Sobolev norms $H^r$), the faster the convergence of this truncated
series to $v$: for all real numbers $r$ and $s$ with $r \le s$, we
have  
\begin{equation}
\label{eq:app-Fourier}
\forall v\in H^s_\#(\Omega), \quad \|v - \Pi_Nv\|_{H^r} \le
\frac{1}{N^{s-r}} \|v\|_{H^s}.
\end{equation}
Let $u_N$ be a solution to the variational problem 
$$
\inf \left\{E(v_N), \; v_N \in \widetilde X_N, \; \int_\Omega v_N^2 = 1
\right\} 
$$
such that $(u_N,u)_{L^2} \ge 0$. 
Using (\ref{eq:app-Fourier}), we obtain
$$
\|u-\Pi_Nu\|_{H^1}  \le  \frac{1}{N^{\sigma+1}} \|u\|_{H^{\sigma+2}},
$$ 
and it therefore follows from the first assertion of
Theorem~\ref{Th:basic} that 
$$
\lim_{N \to \infty} \|u_N-u\|_{H^1} = 0.
$$
We then observe that $u_N$ is solution to the elliptic equation
\begin{equation}
\label{eq:6N} 
-\Delta u_N + \Pi_N \left[ Vu_N+f(u_N^2)u_N \right] = \lambda_N u_N.
\end{equation}
Thus $u_N$ is uniformly bounded in $H^2_\#(\Omega)$, hence in
$L^\infty(\Omega)$, and 
\begin{eqnarray} \label{eq:6Nbis} 
\Delta \left(u_N-u\right) & = & \Pi_N \left( V(u_N-u)+f(u_N^2)u_N-f(u^2)u
\right) \nonumber \\ & & - (I-\Pi_N) (Vu+f(u^2)u) - \lambda_N(u_N-u) -
(\lambda_N-\lambda) u.
\end{eqnarray}
As $(u_N)_{N \in \N}$ in bounded in $L^\infty(\Omega)$ and converges to
$u$ in $H^1_\#(\Omega)$, the right 
hand side of the above equality converges to $0$ in $L^2_\#(\Omega)$,
which implies that $(u_N)_{N \in \N}$ converges to $0$ in
$H^2_\#(\Omega)$, and therefore in $C^0_\#(\Omega)$. In particular,
$u/2 \le u_N \le 2u$ on $\Omega$ for $N$ large enough, so that we can
assume in our analysis, without loss of generality, that $F$
satisfies (\ref{eq:Hyp7}) with $q=0$ and (\ref{eq:Hyp8}) with $r=2$ and $s=0$.
We also deduce from (\ref{eq:6N}) that $u_N$ 
converges to $u$ in $H^{\sigma+2}_\#(\Omega)$.

Besides, the unique solution to (\ref{eq:adjoint}) solves the elliptic
equation 
\begin{eqnarray} \!\!\!\!\!\!\!\!\!\!\!\!\!\!\!\!
&& - \Delta \psi_w + \left( V+f(u^2)+2f'(u^2)u^2-\lambda \right) \psi_w
\nonumber \\ & & \qquad\qquad =
2 \left(\int_\Omega f'(u^2)u^3\psi_w\right) u + w - (w,u)_{L^2} u,
\label{eq:phi_w_Fourier}
\end{eqnarray}
from which we infer that $\psi_{u_N-u} \in H^2_\#(\Omega)$ and
$\|\psi_{u_N-u}\|_{H^2} \le C \|u_N-u\|_{L^2}$. Hence,
$$
\|\psi_{u_N-u}-\Pi_{N}\psi_{u_N-u}\|_{H^1}  \le  \frac{1}{N}
\|\psi_{u_N-u}\|_{H^2} \le  \frac{C}{N} \|u_N-u\|_{L^2}.
$$ 
We therefore deduce from Theorem~\ref{Th:basic} that
\begin{eqnarray} \!\!\!\!\!\!\!\!\!\!\!
\|u_N-u\|_{H^s} & \le & \frac{C}{N^{\sigma+2-s}} \qquad 
\mbox{for } s=0 \mbox{ and } s=1  \label{eq:error_H1_F_0} \\
\!\!\!\!\!\!\!\!\!\!\!
|\lambda_N-\lambda| & \le &  \frac{C}{N^{\sigma+2}}
\label{eq:error_lambda_F} \\
\!\!\!\!\!\!\!\!\!\!\!
\frac \gamma 2 \|u_N-u\|_{H^1}^2 \le E(u_N)-E(u) & \le & C
\|u_N-u\|_{H^1}^2.
\nonumber
\end{eqnarray}
From (\ref{eq:error_H1_F_0}) and the inverse inequality
$$
\forall v_N \in \widetilde X_N, \quad 
\|v_N\|_{H^r} \le 2^{(r-s)/2}  N^{r-s} \|v_N\|_{H^s}, 
$$
which holds true for all $s \le r$ and all $N \ge 1$,
we then obtain using classical arguments that
\begin{equation}  \label{eq:error_H1_F}
\|u_N-u\|_{H^s}  \le  \frac{C}{N^{\sigma+2-s}} \qquad 
\mbox{for all } 0 \le s < \sigma+2.
\end{equation}
The estimate (\ref{eq:error_lambda_F}) is slightly deceptive since, in the case
of a  linear eigenvalue problem (i.e. for $-\Delta u +V u = \lambda u$)
the convergence of the eigenvalues goes twice as fast as the convergence
of the eigenvector in the $H^1$-norm. We are going to prove that this is
also the case for the nonlinear eigenvalue problem under study in this
section, at least under the assumption that $F \in
C^{[\sigma]+2,\sigma-[\sigma]+\epsilon}((0,+\infty),\R)$. 

Let us first come back to (\ref{Cal_lambda}), which we rewrite as,
\begin{equation}  \label{eq:lambda_N}
\lambda_N-\lambda = 
 \langle (A_u-\lambda)(u_N-u ),(u_N-u ) \rangle_{X',X}
+   \int_\Omega w_{u,u_N } (u_N-u )
\end{equation}
with 
$$
w_{u,u_N} = u_N^2 \frac{f(u_N^2)-f(u^2)}{u_N-u} = u_N^2 (u_N+u) 
\frac{f(u_N^2)-f(u^2)}{u_N^2-u^2} .
$$
As $u/2 \le u_N \le 2u$ on $\Omega$ for $N$ large enough, as $u_N$
converges, hence is 
uniformly bounded, in $H^{\sigma+2}_\#(\Omega)$ and as $f \in
C^{[\sigma]+1,\sigma-[\sigma]+\epsilon}([\|u\|_{L^\infty}^2/4,4\|u\|_{L^\infty}^2],\R)$,
we obtain that
$w_{u,u_N}$ is uniformly bounded in  
$H^{\sigma}_\#(\Omega)$ (at least for $N$ large enough). We therefore
infer from (\ref{eq:lambda_N}) that for $N$ large enough
\begin{equation}   \label{eq:lambda_N_2}
|\lambda_N-\lambda| \le C \left( \|u_N -u\|_{H^1}^2 + \|u_N-u\|_{H^{-\sigma}}
\right). 
\end{equation}

Let us now compute the $H^{-r}$-norm of the error for $0 < r \le
\sigma$. Let $w \in
H^{r}_\#(\Omega)$. Proceeding as in
Section~\ref{sec:basic}, we obtain  
\begin{eqnarray}
 \int_\Omega w (u_N-u) & = &
\langle  (E''(u)-\lambda)(u_N-u), \Pi^1_{\widetilde X_N \cap u^\perp} \psi_{w}
\rangle_{X',X}  \nonumber \\
& & + \langle  (E''(u)-\lambda)(u_N-u),\psi_{w}-\Pi^1_{\widetilde X_N \cap u^\perp}\psi_{w}
\rangle_{X',X} \nonumber \\ 
& & + \|u_N-u\|_{L^2}^2 \int_\Omega f'(u^2) u^3 \psi_{w} 
- \frac 12 \|u_N-u \|_{L^2}^2 \int_\Omega uw, \label{eq:intom}
\end{eqnarray}
where $\Pi^1_{\widetilde X_N \cap u^\perp}$ denotes the
orthogonal projector on $\widetilde X_N \cap u^\perp$ for
the $H^1$ inner product.
We then get from (\ref{eq:phi_w_Fourier}) that $\psi_w$ is in
$H^{r+2}_\#(\Omega)$ and satisfies 
\begin{equation} \label{eq:borne_psiw}
\|\psi_w\|_{H^{r+2}} \le C \|w\|_{H^r},
\end{equation}
for some constant $C$ independent of $w$. 

Combining (\ref{eq:Euu-lambda}), (\ref{eq:new37}), (\ref{eq:minuperp}), 
(\ref{eq:error_H1_F}),
(\ref{eq:lambda_N}), (\ref{eq:intom}) and (\ref{eq:borne_psiw}), we
obtain that there exists a constant $C \in \R_+$ such that for all $N
\in \N$ and all $w \in H^{r}_\#(\Omega)$,
\begin{eqnarray*}
\int_\Omega w (u_N-u) & \le & C' \left( \|u_N-u\|_{L^2}^2 + N^{-(r+1)} 
\|u_N-u\|_{H^1} \right) \|w\|_{H^r} \\
 & \le & \frac{C}{N^{\sigma+2+r}} \|w\|_{H^{r}} .
\end{eqnarray*}

Therefore
\begin{equation} \label{eq:Hm1boundFourier}
\|u_N-u\|_{H^{-r}} = \sup_{w \in H^{r}_\#(\Omega) \setminus
  \left\{0\right\}} \frac{\dps \int_\Omega w (u_N-u)}{\|w\|_{H^{r}}}
\le \frac{C}{N^{\sigma+2+r}},
\end{equation}
for some constant $C \in \R_+$ independent of $N$.
Using (\ref{eq:error_H1_F}) and (\ref{eq:lambda_N_2}), we end up with 
$$
|\lambda_N-\lambda| \le  \frac{C}{N^{2(\sigma+1)}}.
$$

We can summarize the results obtained in this
section in the following theorem.

\medskip

\begin{thm} \label{Th:Fourier}
Assume that $V \in H^\sigma_\#(\Omega)$ for some $\sigma > d/2$ and
that the function $F$ satisfies (\ref{eq:Hyp6})-(\ref{eq:Hyp7p}) and is
in $C^{[\sigma]+1,\sigma-[\sigma]+\epsilon}((0,+\infty),\R)$. Then $(u_N)_{N \in
  \N}$ converges to $u$ in $H^{\sigma+2}_\#(\Omega)$ and there exists $C \in
\R_+$ such that for all $N \in \N$, 
\begin{eqnarray}
 \!\!\!\!\!\!\!\!\!\!\!
\|u_N-u\|_{H^s} & \le & \frac{C}{N^{\sigma+2-s}} \qquad 
\mbox{for all } -\sigma \le s < \sigma+2 
\label{eq:H1boundFourier} \\
 \!\!\!\!\!\!\!\!\!\!\!
|\lambda_N-\lambda| & \le &  \frac{C}{N^{\sigma+2}} \nonumber
 \\
 \!\!\!\!\!\!\!\!\!\!\!
\frac \gamma 2 \|u_N-u\|_{H^1}^2 \le E(u_N)-E(u) & \le & C \|u_N-u\|_{H^1}^2.
\end{eqnarray}
If, in addition, $F \in
C^{[\sigma]+2,\sigma-[\sigma]+\epsilon}((0,+\infty),\R)$, then 
\begin{equation} \label{eq:lambdaboundFourier}
|\lambda_N-\lambda| \le   \frac{C}{N^{2(\sigma+1)}}. 
\end{equation} 
\end{thm} 

\medskip

In order to evaluate the quality of the error bounds obtained in
Theorem~\ref{Th:Fourier}, we have performed numerical tests with $\Omega
= (0,2\pi)$, $V(x)=\sin(|x-\pi|/2)$ and $F(t^2)=t^{2}/2$. The Fourier
coefficients of the potential $V$ are given by 
\begin{equation} \label{eq:coeff_Fourier}
\widehat V_k = - \frac{1}{\sqrt{2\pi}} \frac{1}{|k|^2- \frac 14},
\end{equation}
from which we deduce that $V \in H^\sigma_\#(0,2\pi)$ for all $\sigma <
3/2$. It can be see on Figure~1 that  $\|u_N-u\|_{H^1}$,
  $\|u_N-u\|_{L^2}$, $\|u_N-u\|_{H^{-1}}$, and $|\lambda_N-\lambda|$
  decay respectively as $N^{-2.67}$, $N^{-3.67}$, $N^{-4.67}$ and
  $N^{-5}$ (the reference values for $u$ and $\lambda$ are those obtained
  for $N=65$). These results are in good agreement with the upper
  bounds (\ref{eq:H1boundFourier}) (for $s=1$ and $s=0$), 
  (\ref{eq:Hm1boundFourier}) (for $r=1$) and
  (\ref{eq:lambdaboundFourier}), which 
  respectively decay as $N^{-2.5+\epsilon}$, $N^{-3.5+\epsilon}$,
  $N^{-4.5+\epsilon}$ and $N^{-5+\epsilon}$, for $\epsilon > 0$
  arbitrarily small. 

\medskip

\begin{figure}[h] \label{fig:Fourier}
\centering
\psfig{figure=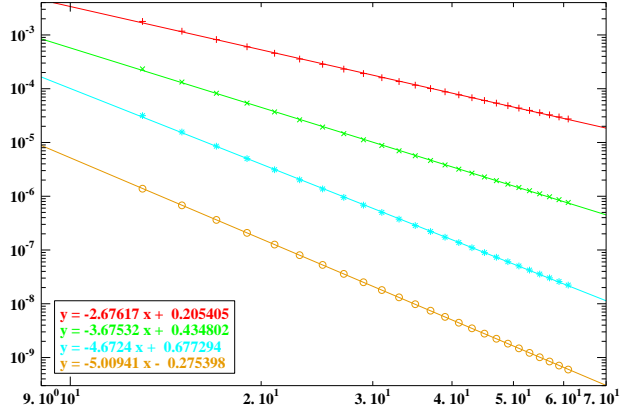,height=6truecm}
\caption{Numerical errors $\|u_N-u\|_{H^1}$ ($+$),
  $\|u_N-u\|_{L^2}$ ($\times$), $\|u_N-u\|_{H^{-1}}$ ($\ast$), and
  $|\lambda_N-\lambda|$ ($\circ$), as
functions of $2N+1$ (the dimension of $\widetilde X_N$) in log scales.}
\end{figure}

\section{Finite element discretization} 
\label{sec:FE}
In this section, we consider the problem
\begin{equation} \label{eq:F.E}
\inf \left\{ E(v), \; v \in X, \; \int_\Omega v^2=1 \right\},
\end{equation}
where
\begin{eqnarray*}
&& \Omega \textrm{ is a rectangular brick of  } \R^d,  \quad \mbox{with
  $d=1$, $2$ or $3$,} \\ 
&& X=H^1_0(\Omega), 
\nonumber \\
&& E(v) = \frac 12 \int_\Omega |\nabla v|^2 + \frac 12 \int_\Omega V v^2
+ \frac {1}{2} \int_\Omega F(v^2).
\nonumber
\end{eqnarray*}
We assume that $V \in L^2(\Omega)$ and
that the function $F$ satisfies
(\ref{eq:Hyp6})-(\ref{eq:Hyp7p}), as well as (\ref{eq:Hyp8}) for some $1
< r \le 2$ and $0 \le r+s \le 3$. Throughout this section, we denote by
$u$ the unique positive solution of (\ref{eq:F.E}) and by $\lambda$ the
corresponding Lagrange multiplier.
 
In the non periodic case considered here, a classical variational
approximation of~(\ref{eq:min_pb_u}) is provided by the finite element
method. We
consider a family of regular triangulations $({\cal T}_h)_h$ of
$\Omega$. This means, in the case when $d=3$ for instance, that for each
$h > 0$, ${\cal T}_h$ is a collection of tetrahedra such that
 \begin{itemize}
\item $\overline\Omega$ is the union of all the elements of ${\cal T}_h$;
 \item the intersection of two different elements of ${\cal T}_h$ is
   either empty, a vertex, a whole edge, or a whole face of both of them;
 \item the ratio of the diameter $h_K$ of any element $K$ of ${\cal T}_h$ to the
diameter of its inscribed sphere is smaller than a constant independent of $h$.
\end{itemize}
As usual, $h$ denotes the maximum of the diameters $h_K$, $K\in {\cal
  T}_h$. The parameter of the discretization then is $\delta=h > 0$. For
each $K$ in ${\cal T}_h$ and each nonnegative integer $k$, we denote by
$\mathbb P_k(K)$ the space of the restrictions to $K$ of the polynomials
with $d$ variables and total degree lower or equal to $k$.

The finite element space $X_{h,k}$ constructed from ${\cal T}_h$ and
$\mathbb P_k(K)$ is the space of all continuous functions on $\Omega$
vanishing on $\partial\Omega$ such that their restrictions to any
element $K$ of ${\cal T}_h$ belong to $\mathbb P_k(K)$. Recall that
$X_{h,k} \subset H^1_0(\Omega)$ as soon as $k \ge 1$.

We denote by $\pi_{h,k}^0$ and $\pi_{h,k}^1$ the orthogonal projectors
on $X_{h,k}$ for the $L^2$ and $H^1$ inner products respectively.
The following estimates are classical (see e.g. \cite{Ern}): there
exists $C \in \R_+$ such 
that for all $r \in \N$ such that $1 \le r \le k+1$,
\begin{eqnarray}
\forall \phi \in H^r(\Omega) \cap H^1_0(\Omega), & \quad &
\|\phi-\pi_{h,k}^0\phi\|_{L^2} \le C h^r \|\phi\|_{H^r}, 
\nonumber \\
\forall \phi \in H^r(\Omega) \cap H^1_0(\Omega), & \quad &
\|\phi-\pi_{h,k}^1\phi\|_{H^1} \le C h^{r-1} \|\phi\|_{H^r}. \label{eq:proj_H1}
\end{eqnarray}
Let $u_{h,k}$ be a solution to the variational problem 
$$
\inf \left\{E(v_{h,k}), \; v_{h,k} \in X_{h,k}, \; \int_\Omega v_{h,k}^2
  = 1 \right\} 
$$
such that $(u_{h,k},u)_{L^2} \ge 0$. In this setting, we obtain the
following {\it a priori} error estimates.

\begin{thm} \label{Th:FE}
 Assume that $V \in L^2(\Omega)$ and
that the function $F$ satisfies (\ref{eq:Hyp6}), (\ref{eq:Hyp7}) for
$q=1$, (\ref{eq:Hyp7p}), and 
(\ref{eq:Hyp8}) for some $1 < r \le 2$ and $0 \le r+s \le 3$.
Then there exists $h_0 > 0$ and $C \in \R_+$ such that for all $0 < h
\le h_0$,
\begin{eqnarray}
\|u_{h,1}-u\|_{H^1} & \le &  C \, h \label{eq:estim_P1_H1} \\
\|u_{h,1}-u\|_{L^2} & \le &  C \, h^2 \label{eq:estim_P1_L2} \\
|\lambda_{h,1}-\lambda| & \le & C \, h^2 \label{eq:estim_P1_lambda} \\
\frac \gamma 2 \|u_{h,1}-u\|_{H^1}^2 \le E(u_{h,1})-E(u) & \le & C \,
h^2.  \label{eq:estim_P1_energy} 
\end{eqnarray}
If in addition, $V \in H^1(\Omega)$, $F$ satisfies (\ref{eq:Hyp8}) for
$r=2$ and is such that $F \in C^3((0,+\infty),\R)$ and 
$F''(t)t^{1/2}$ and $F'''(t)t^{3/2}$ are bounded in the vicinity of $0$, 
then there exists $h_0 > 0$ and $C \in \R_+$ such that for all $0 < h
\le h_0$, 
\begin{eqnarray}
\|u_{h,2}-u\|_{H^1} & \le &  C \, h^2 \label{eq:estim_P2_H1} \\
\|u_{h,2}-u\|_{L^2} & \le &  C \, h^3  \label{eq:estim_P2_L2} \\
|\lambda_{h,2}-\lambda| & \le & C \, h^4 \label{eq:estim_P2_lambda} \\
\frac \gamma 2 \|u_{h,2}-u\|_{H^1}^2 \le E(u_{h,2})-E(u) 
& \le & C \, h^4. \label{eq:estim_P2_energy}
\end{eqnarray}
\end{thm} 

\begin{proof}
As $\Omega$ is a rectangular brick, $V$ satisfies (\ref{eq:Hyp3}) and
$F$ satisfies (\ref{eq:Hyp6})-(\ref{eq:Hyp7p}), we
have $u \in H^2(\Omega)$. We then use the fact
that $\psi_{u_{h,k}-u}$ is solution to 
\begin{eqnarray}
& & - \Delta \psi_{u_{h,k}-u} +
(V+f(u^2)+2f'(u^2)u^2-\lambda)\psi_{u_{h,k}-u} \nonumber  
\\ && \qquad \qquad =
2 \left(\int_\Omega f'(u^2)u^3\psi_{u_{h,k}-u}\right) u + (u_{h,k}-u) -
(u_{h,k}-u,u)_{L^2} u, \qquad\qquad
\nonumber
\end{eqnarray}
to establish that $\psi_{u_{h,k}-u} \in H^2(\Omega) \cap H^1_0(\Omega)$ and that
\begin{equation} \label{eq:bound_psi_hk}
\| \psi_{u_{h,k}-u} \|_{H^2} \le C \| u_{h,k}-u \|_{L^2}
\end{equation}
for some constant $C$ independent of $h$ and $k$.
The estimates (\ref{eq:estim_P1_H1})-(\ref{eq:estim_P1_energy}) then are
directly consequences of Theorem~\ref{Th:basic},
(\ref{eq:error_L2_simpler}), (\ref{eq:proj_H1}) and
(\ref{eq:bound_psi_hk}). 

Under the additional assumptions that $V \in H^1(\Omega)$, we obtain by
standard elliptic 
regularity arguments that $u \in H^3(\Omega)$. The $H^1$ and $L^2$
estimates (\ref{eq:estim_P2_H1}) and (\ref{eq:estim_P2_L2}) immediately
follows from Theorem~\ref{Th:basic}, (\ref{eq:error_L2_simpler}),
(\ref{eq:proj_H1}) and (\ref{eq:bound_psi_hk}). We also have
$$
|\lambda_{2,h}-\lambda| \le C h^3
$$
for a constant $C$ independent of $h$.
In order to prove (\ref{eq:estim_P2_lambda}), we proceed as in
Section~\ref{sec:Fourier}. We start from the equality
$$
\lambda_{2,h}-\lambda = 
 \langle (A_u-\lambda)(u_{2,h}-u ),(u_{2,h}-u ) \rangle_{X',X}
+   \int_\Omega \widetilde w^h (u_{2,h}-u )
$$
where
$$
\widetilde w^h = u_{2,h}^2 \frac{f(u_{2,h}^2)-f(u^2)}{u_{2,h}-u}.
$$
We now claim that $u_{h,2}$ converges to $u$ in
$L^\infty(\Omega)$ when $h$ goes to zero. To establish this result, we first
remark that 
$$
\|u_{h,2}-u\|_{L^\infty} \le 
\|u_{h,2}- {\cal I}_{h,2} u \|_{L^\infty} + \|{\cal I}_{h,2} u-u\|_{L^\infty},
$$ 
where ${\cal I}_{h,2}$ is the interpolation projector on $X_{h,2}$. As $u \in
H^3(\Omega) \hookrightarrow C^1(\overline{\Omega})$, we have
$$
\lim_{h \to 0^+} \|{\cal I}_{h,2} u-u\|_{L^\infty} = 0.
$$
On the other hand, using the inverse inequality 
$$
\exists C \in \R_+ \mbox{ s.t. } \forall 0 < h \le h_0, \; \forall v_h
\in X_{h,2}, \quad \|v_{h,2}\|_{L^\infty} \le C \rho(h) \|v_{h,2}\|_{H^1}, 
$$
with $\rho(h)=1$ if $d=1$, $\rho(h)=1+\ln h$ if $d=2$ and
$\rho(h)=h^{-1/2}$ if $d=3$ (see \cite{Ern} for instance), we obtain
\begin{eqnarray*}
\|u_{h,2}- {\cal I}_{h,2} u\|_{L^\infty} & \le & C \rho(h) 
\|u_{h,2}- {\cal I}_{h,2} u \|_{H^1}  \\
& \le &  C \rho(h) \left(  \|u_{h,2}-u \|_{H^1} + \|u-{\cal I}_{h,2} u\|_{H^1}
\right) \\
& \le & C' \, \rho(h) \, h^2 \; \mathop{\longrightarrow}_{h \to 0^+} \; 0.
\end{eqnarray*}
Hence the announced result. This implies in particular that $\widetilde
w^h$ is bounded in $H^1(\Omega)$, uniformly in $h$. 
Consequently, there exists $C \in \R_+$ such that for all $0 < h \le
h_0$, 
\begin{equation} \label{eq:estim_lambda_FE}
|\lambda_{h,2}-\lambda| \le C \left( \|u_{h,2}-u\|_{H^1}^2
+ \|u_{h,2}-u\|_{H^{-1}} \right).
\end{equation}
To estimate the $H^{-1}$-norm of $u_{h,2}-u$, we write that for all $w
\in H^1_0(\Omega)$, 
\begin{eqnarray}
\!\!\!\!\!\!\!\!\!\!\!\! \int_\Omega w (u_{h,2}-u)  & = &
\langle
(E''(u)-\lambda)(u_{h,2}-u), \pi^1_{X_{h,2}\cap u^\perp} \psi_w \rangle_{X',X}
 \nonumber \\
& & + \langle
(E''(u)-\lambda)(u_{h,2}-u), \psi_w - \pi^1_{X_{h,2}\cap u^\perp} 
 \psi_w \rangle_{X',X}  \nonumber \\ 
& & + \|u_{h,2}-u\|_{L^2}^2 \int_\Omega f'(u^2) u^3  \psi_w 
- \frac 12 \|u_{h,2}-u \|_{L^2}^2 \int_\Omega u  w,
\nonumber
\end{eqnarray}
where $\psi_w$ is solution to 
\begin{eqnarray}
&& - \Delta  \psi_w +
(V+f(u^2)+2f'(u^2)u^2-\lambda)  \psi_w  \nonumber \\ && 
\qquad \qquad =
2 \left(\int_\Omega f'(u^2)u^3 \psi_w\right) u +  w -
( w,u)_{L^2} u,\qquad \qquad
\label{eq:eq_sur_psit}
\end{eqnarray}
and where $\pi^1_{X_{h,2}\cap u^\perp}$ denotes the orthogonal projector
on $X_{h,2}\cap u^\perp$ for the $H^1$ inner product.
Using the assumptions that $V \in H^1(\Omega)$, $F \in
C^3((0,+\infty),\R)$, and $F''(t)t^{1/2}$ and $F'''(t)t^{3/2}$ are
bounded in the vicinity of $0$, we deduce from
(\ref{eq:eq_sur_psit}) that $\psi_w$ is in $H^3(\Omega)$ and
that there exists $C \in \R_+$ such that for all $w \in H^1_0(\Omega)$
and all $0 < h \le h_0$, 
$$
\| \psi_w\|_{H^3} \le C \| w\|_{H^1}.
$$
We therefore obtain the inequality 
\begin{equation} \label{eq:bound_psih}
\|  \psi_w-\pi^1_{h,2}  \psi_w\|_{H^1} \le C 
h^2 \|w\|_{H^1}, 
\end{equation}
where the constant $C$ is independent of $h$. 
 
Putting together (\ref{eq:Hyp8}) (for $r=2$), (\ref{eq:Euu-lambda}),
(\ref{eq:new37}), (\ref{eq:minuperp}),  
(\ref{eq:proj_H1}), (\ref{eq:estim_P2_H1}), (\ref{eq:estim_P2_L2}) and
(\ref{eq:bound_psih}), we get 
$$
\|u_{h,2}-u\|_{H^{-1}} = \sup_{w \in H^1_0(\Omega) \setminus
  \left\{0\right\}} \frac{\int_\Omega w(u_{h,2}-u)}{\|w\|_{H^1}} \le C
\, h^4.
$$
Combining with (\ref{eq:estim_P2_H1}) and (\ref{eq:estim_lambda_FE}), we
end up with (\ref{eq:estim_P2_lambda}). Lastly, we deduce
(\ref{eq:estim_P2_energy}) from the equality
\begin{eqnarray*}
E(u_{h,2}) - E(u) & = &\frac 12 \langle (A_u-\lambda) (u_{h,2}-u)
,(u_{h,2}-u)  \rangle_{X',X}  \\ 
& & + \frac 12 \int_\Omega F\left(  u^2 + (u_{h,2} ^2-u^2) \right)
- F\left( u^2 \right) - f\left(  u^2 \right) (u_{h,2} ^2-u^2),
\end{eqnarray*}
Taylor expanding the integrand and exploiting the boundedness of the
function $F''(t)t^{1/2}$ in the vicinity of $0$.   
\end{proof}

Numerical results for the case when $\Omega = (0,\pi)^2$, $V(x_1,x_2) =
x_1^2+x_2^2$ and $F(t^2)=t^2/2$ are reported on Figure~2. The agreement
with the error estimates obtained in Theorem~\ref{Th:FE} is good for the 
$\mathbb P_1$ approximation and excellent for the $\mathbb P_2$
approximation.  

\begin{figure}[h] \label{fig:FE}
\centering
\begin{tabular}{c}
\psfig{figure=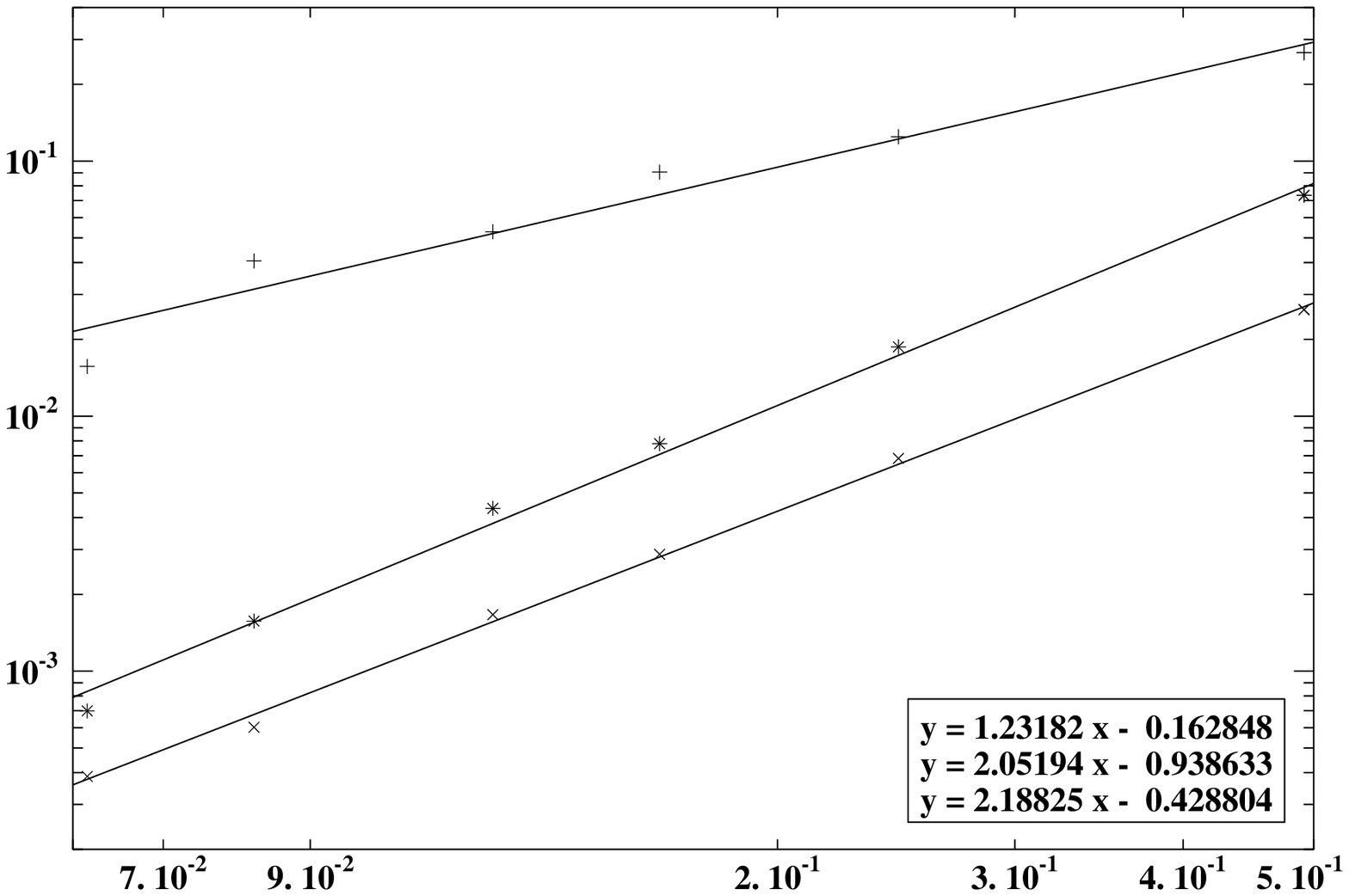,height=7truecm} \\
\psfig{figure=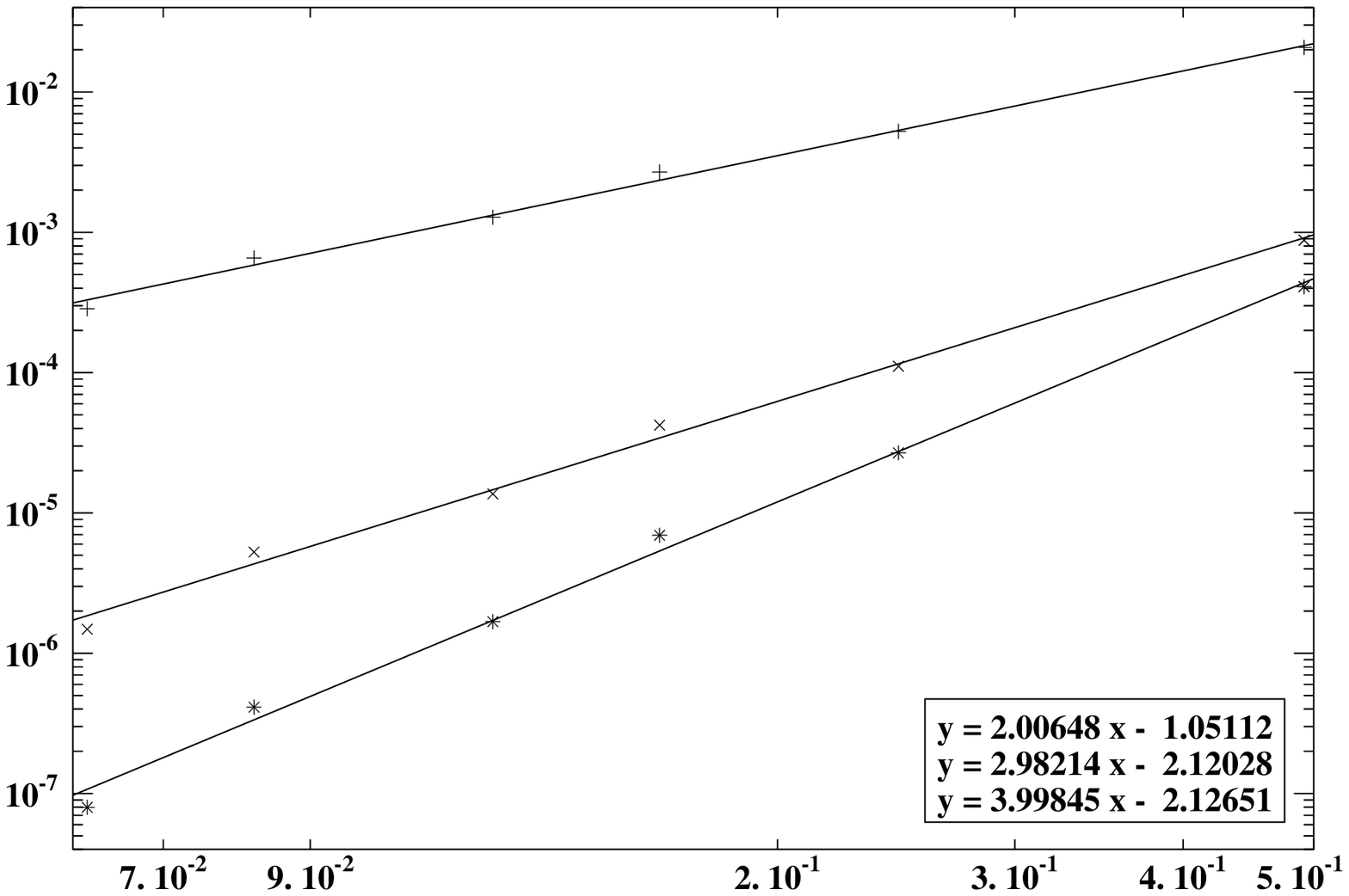,height=7truecm} 
\end{tabular}
\caption{Errors $\|u_{h,k}-u\|_{H^1}$ ($+$), $\|u_{h,k}-u\|_{L^2}$
  ($\times$) and
  $|\lambda_{h,k}-\lambda|$ ($\ast$) for the $\mathbb P_1$ ($k=1$, top) and
  $\mathbb P_2$ ($k=2$, bottom) approximations as a function of $h$ in
  log scales.}
\end{figure}

\section{The effect of numerical integration}
\label{sec:integration}

Let us now address one further consideration that is related to the
practical implementation of the method, and more precisely to the
numerical integration of the nonlinear term. For simplicity, we focus on
the case when $A = 1$.

From a practical viewpoint, the solution $(u_\delta,\lambda_\delta)$ to the
nonlinear eigenvalue problem (\ref{eq:14}) can be computed iteratively, using
for instance the optimal damping algorithm~\cite{ODA1,ODA2,GP}. At the
$p^{\rm th}$ iteration ($p \ge 1$), the ground state
$(u_\delta^p,\lambda_\delta^p) \in  X_\delta \times \R$ of some {\it
  linear}, finite dimensional, eigenvalue problem of the form   
\begin{equation}
\label{eq:6SCF}
\int_\Omega \overline{\nabla u^p_\delta} \cdot \nabla v_\delta + \int_\Omega 
\left(V+f(\widetilde\rho_\delta^{p-1})\right) \, \overline{u_\delta^{p}}
\, v_\delta = 
  \lambda_\delta^{p}\int_\Omega 
  \overline{u_\delta^{p}} v_\delta,\quad \forall v_\delta \in X_\delta, 
\end{equation}
has to be computed. In the optimal damping algorithm, the density
$\widetilde\rho_\delta^{p-1}$ is a convex linear combination of the
densites $\rho_\delta^q = |u^q_\delta|^2$, for $0 \le q \le p-1$. 
Solving (\ref{eq:6SCF}) amounts to finding the lowest eigenelement of
the matrix $H^p$ with entries 
\begin{equation}
\label{eq:6SCFmat}
H^p_{kl} := \int_\Omega \overline{\nabla \phi_k} \cdot \nabla \phi_l +
\int_\Omega  
V \, \overline{\phi_k} \,  \phi_l 
+ \int_\Omega f(\widetilde\rho_\delta^{p-1}) \, \overline{\phi_k} \, \phi_l,
\end{equation}
where $(\phi_k)_{1 \le k \le \mbox{dim}(X_\delta)}$ stands for the
canonical basis of $X_\delta$. 

In order to evaluate the last two terms of the right-hand side
of (\ref{eq:6SCFmat}), numerical integration has to be resorted to. In
the finite element approximation of (\ref{eq:F.E}), it is generally made
use of a numerical quadrature formula over each triangle (2D) or
tetrahedron (3D) based on Gauss points. In the Fourier approximation of
the periodic problem (\ref{eq:Fourier}), the terms 
$$
\int_\Omega  V \, \overline{e_k} \, e_l \quad \mbox{and} \quad 
 \int_\Omega  f(\widetilde\rho_\delta^{p-1}) \, \overline{e_k} \, e_l,
$$
which are in fact, up to a multiplicative constant, the $(k-l)^{\rm th}$
Fourier coefficients of $V$ and  
$f(\widetilde\rho_\delta^{p-1})$ respectively, are evaluated by Fast
Fourier Transform (FFT), using an integration grid which may be
different from the natural discretization grid
$$
\left\{\left(\frac{2\pi}{2N+1}j_1, \cdots, \frac{2\pi}{2N+1}j_d,
 \right), \; 0 \le j_1,\cdots,j_d \le 2N \right\}
$$ 
associated with $\widetilde X_N$. This raises the question of
the influence of the numerical 
integration on the convergence results obtained in
Theorems~~\ref{Th:basic}, \ref{Th:Fourier} and~\ref{Th:FE}.

\begin{remark} \label{rem:doublegrid}
In the case of the periodic problem considered in
Section~\ref{sec:Fourier} and when $F(t) = ct^2$ for some $c > 0$, the
last term of the right-hand side 
of (\ref{eq:6SCFmat}) can be computed exactly (up to round-off errors)
by means of a Fast Fourier Transform (FFT) on an integration grid twice
as fine as the 
discretization grid. This is due to the fact that the function 
$\widetilde\rho_\delta^{p-1} \, \overline{e_k} \, e_l$ belongs to the
space $\mbox{Span} \{e_n \, | \, |n|_* \le 4N \}$.
An analogous property is used in the evaluation of the Coulomb term in the
numerical simulation of the Kohn-Sham equations for periodic systems.
\end{remark} 

In the sequel, we focus on the simple case when $d=1$,
$\Omega=(0,2\pi)$, $X = H^1_\#(0,2\pi)$, and
$$
E(v) = \frac 12 \int_0^{2\pi} |v'|^2 + \frac 12 \int_0^{2\pi} V v^2 +
\frac 14 \int_0^{2\pi} |v|^4
$$
with $V \in H^\sigma_\#(0,2\pi)$ for some $\sigma > 1/2$. More difficult
cases will be addressed elsewhere~\cite{CCM2}.  

In view of Remark~\ref{rem:doublegrid}, we consider an integration grid 
$$
\frac{2\pi}{N_g}\Z \cap [0,2\pi) =
\left\{0,\frac{2\pi}{N_g},\frac{4\pi}{N_g}, 
    \cdots, \frac{2\pi(N_g-1)}{N_g} \right\},
$$
with $N_g \ge 4N+1$ for which we have
$$
\forall v_N \in \widetilde X_N, \quad \int_0^{2\pi} |v_N|^4 = \frac{2\pi}{N_g}
\sum_{r \in \frac{2\pi}{N_g}\Z \cap [0,2\pi)} |v_N(r)|^4,
$$
and for all $\rho \in \widetilde X_{2N}$,
\begin{equation} \label{eq:intu4}
\forall |k|,|l|\le N, \quad
\int_0^{2\pi} \rho \, \overline{e_k} \, e_l = 
 \frac{1}{N_g}
\sum_{r \in \frac{2\pi}{N_g}\Z \cap [0,2\pi)}  \rho(r)
e^{-i(k-l)r} = \widehat{\rho^{{\rm FFT}}_{k-l}},
\end{equation}
where $\widehat{\rho^{{\rm FFT}}_{k-l}}$ is the $(k-l)^{\rm th}$
coefficient of the discrete Fourier transform of $\rho$. Recall that
if $\phi = \sum_{g \in \Z} \widehat \phi_g \, e_g \in C^0_\#(0,2\pi)$,
the discrete Fourier transform of $\phi$ is the $N_g\Z$-periodic
sequence $(\widehat{\phi^{{\rm FFT}}_{g}})_{g \in \Z}$ defined by 
$$
\forall g \in \Z, \quad \widehat{\phi^{{\rm FFT}}_{g}}
= \frac{1}{N_g} \sum_{r \in \frac{2\pi}{N_g}\Z \cap [0,2\pi)}  \phi(r)
e^{-igr}.
$$
We now introduce the subspaces $W_M$ for $M \in \N^\ast$ such that
$W_{M} = \widetilde X_{(M-1)/2}$ if $M$ is odd and $W_{M} = \widetilde X_{M/2-1}
\oplus \C (e_{M/2}+e_{-M/2})$ is $M$ is even (note that
$\mbox{dim}(W_{M})=M$ for all $M \in \N^\ast$). It is then possible to define
an interpolation projector ${\cal I}_{N_g}$ from $C^0_\#(0,2\pi)$ onto
$W_{N_g}$ by 
$$
\forall x \in \frac{2\pi}{N_g}\Z \cap [0,2\pi), \quad [{\cal
  I}_{N_g}(\phi)](x) = \phi(x).
$$
The expansion of ${\cal I}_{N_g}(\phi)$ in the canonical basis of
$W_{N_g}$ is given by
$$
{\cal I}_{N_g}(\phi) \; = \; \left| \begin{array}{lll}
(2\pi)^{1/2} \dps \sum_{|g| \le (N_g-1)/2} \widehat{\phi^{{\rm
      FFT}}_{g}} \, e_g & \quad & (N_g \mbox{ odd}), \\
(2\pi)^{1/2} \dps \sum_{|g| \le N_g/2-1} \widehat{\phi^{{\rm
      FFT}}_{g}} \, e_g +  (2\pi)^{1/2} \widehat{\phi^{{\rm
      FFT}}_{N_g/2}} \, \left( \frac{e_{N_g/2}+e_{-N_g/2}}{2} \right) &
\quad & (N_g \mbox{ even}). 
\end{array} \right.
$$
Under the condition that $N_g \ge 4N+1$, the following property holds:
for all $\phi \in C^0_\#(0,2\pi)$,  
$$
\forall |k|,|l| \le N, \quad \int_0^{2\pi} {\cal I}_{N_g}(\phi) \,
\overline{e_k} \, e_l  = \widehat{\phi^{{\rm FFT}}_{k-l}}.
$$
It is therefore possible, in the particular case considered here, to
efficiently evaluate the entries of the matrix $H^p$ using the formula
\begin{eqnarray}
H^p_{kl} & := & \int_0^{2\pi} \overline{e_k'} \cdot e_l' +
\int_0^{2\pi} V \, \overline{e_k} \,  e_l 
+ \int_0^{2\pi} \widetilde\rho_N^{p-1}  \overline{e_k} \, e_l \nonumber \\
& \simeq & |k|^2 \delta_{kl} + \widehat{V^{{\rm FFT}}_{k-l}} + 
\widehat{[\widetilde\rho_{N}^{p-1}]^{\rm FFT}_{k-l}}, \label{eq:approx_FFT}
\end{eqnarray}
and resorting to Fast Fourier Transform (FFT) algorithms to compute the
discrete Fourier transforms.
Note that only the second term is computed approximatively. The
third term is computed exactly since, at each iteration,
$\widetilde\rho_N^{p-1}$ belongs to $\widetilde X_{2N}$ (see
Eq.~(\ref{eq:intu4})). Of course, this situation is specific to the
nonlinearity $F(t)=t^2/2$ considered here.

Using the approximation formula (\ref{eq:approx_FFT}) amounts to replace
the original problem 
\begin{equation} \label{eq:uN}
\inf \left\{E(v_N), \; v_N \in \widetilde X_N, \; \int_0^{2\pi} |v_N|^2 = 1 \right\},
\end{equation}
with the approximate problem
\begin{equation} \label{eq:uNNg}
\inf \left\{E_{N_g}(v_N), \; v_N \in \widetilde X_N , \; \int_0^{2\pi} |v_N|^2 = 1 \right\},
\end{equation}
where
$$
E_{N_g}(v_N) = \frac 12 \int_0^{2\pi} |v_N'|^2 + \frac 12 \int_0^{2\pi}
{\cal I}_{N_g}(V) v_N^2 + \frac 14 \int_0^{2\pi} |v_N|^4.
$$
Let us denote by $u_N$ a solution of (\ref{eq:uN})
such that $(u_N,u)_{L^2} \ge 0$ and by $u_{N,N_g}$ a solution to
(\ref{eq:uNNg}) such that $(u_{N,N_g},u)_{L^2} \ge 0$. It is easy to
check that $u_{N,N_g}$ is bounded in $H^1_\#(0,2\pi)$ uniformly in $N$
and $N_g$.

Besides, we know from Theorem~\ref{Th:Fourier} that $(u_N)_{N \in \N}$
converges to $u$ in $H^1_\#(0,2\pi)$, hence in $L^\infty_\#(2,\pi)$,
when $N$ goes to infinity. This implies that the sequence
$(A_u-A_{u_N})_{N \in \N}$ converges to $0$ in operator
norm. Consequently, for all $N$ large enough and all $N_g$ such that
$N_g \ge 4N+1$,
\begin{eqnarray*}
\frac{\gamma}4 \|u_{N,N_g}-u_N\|_{H^1}^2 & \le & E(u_{N,N_g})-E(u_N) \\
& \le &  E_{N_g}(u_{N,N_g})-E_{N_g}(u_N) \\ & & + \int_0^{2\pi} 
(V-{\cal I}_{N_g}(V)) \left(|u_{N,N_g}|^2-|u_N|^2\right) \\
& \le & \int_0^{2\pi} 
(V-{\cal I}_{N_g}(V)) \left(|u_{N,N_g}|^2-|u_N|^2\right) \\
& \le & C \|\Pi_{2N}(V-{\cal I}_{N_g}(V))\|_{L^2} \|u_{N,N_g}-u_N\|_{H^1},
\end{eqnarray*} 
where we have used the fact that 
$\left(|u_{N,N_g}|^2-|u_N|^2\right) \in \widetilde X_{2N}$. Therefore, 
\begin{equation} \label{eq:uNNgVIV}
\|u_{N,N_g}-u_N\|_{H^1} \le C \|\Pi_{2N}(V-{\cal I}_{N_g}(V))\|_{L^2},
\end{equation}
for a constant $C$ independent of $N$ and $N_g$. 
Likewise,
\begin{eqnarray*}
\lambda_{N,N_g}-\lambda_N & = & \langle (A_{u_N}-\lambda_N)
(u_{N,N_g}-u_N), (u_{N,N_g}-u_N) \rangle_{X',X} \\ & & + \int_0^{2\pi} 
(V-{\cal I}_N(V)) |u_{N,N_g}|^2 \\
& & + \int_0^{2\pi} |u_{N,N_g}|^2 (u_{N,N_g}+u_N)(u_{N,N_g}-u_N),
\end{eqnarray*}
from which we deduce, using (\ref{eq:uNNgVIV}),
$$
|\lambda_{N,N_g}-\lambda_N| \le C \|\Pi_{2N}(V-{\cal I}_{N_g}(V))\|_{L^2}. 
$$
An error analysis of
the interpolation operator ${\cal I}_{N_g}$ is given in~\cite{CHQZ}: for
all non-negative real numbers $0 \le r \le s$ with $s > 1/2$ (for
$d=1$),
$$
\|\varphi - {\cal I}_{N_g}(\varphi)\|_{H^r} \le \frac{C}{N_g^{s-r}}
\|\varphi\|_{H^s},\quad \forall\varphi\in H^s_\#(0,2\pi). 
$$
Thus,
\begin{equation} \label{eq:estimeIV}
\|\Pi_{2N}(V-{\cal I}_{N_g}(V))\|_{L^2} \le 
\|V - {\cal I}_{N_g}(V)\|_{L^2} \le \frac{C}{N_g^{\sigma}},
\end{equation}
and the above inequality provides the following estimates:
\begin{eqnarray} 
\|u_{N,N_g}-u\|_{H^1} & \le & C \left(N^{-\sigma-1}+N_g^{-\sigma}\right)
\label{eq:erroruNNgL2} \\
\|u_{N,N_g}-u\|_{L^2} & \le & C \left(N^{-\sigma-2}+N_g^{-\sigma}\right)
\label{eq:NNN} \\
|\lambda_{N,N_g}-\lambda| & \le & C
\left(N^{-2\sigma-2}+N_g^{-\sigma}\right), \label{eq:errorlambdaNNg} 
\end{eqnarray}
for a constant $C$ independent of $N$ and $N_g$. The first component of
the error bound (\ref{eq:erroruNNgL2}) corresponds to the error
$\|u_N-u\|_{H^1}$ while the second component corresponds to the
numerical integration error $\|u_{N,N_g}-u_N\|_{H^1}$ (the same remark
applies to the error bounds (\ref{eq:NNN}) and (\ref{eq:errorlambdaNNg})).

\medskip

It is classical that for the norm $\|\varphi - {\cal
  I}_{N_g}\varphi\|_{H^{r}}$ for $r < 0$ is in general of the same order
of magnitude as $\|\varphi - {\cal I}_{N_g}\varphi\|_{L^2}$. As the
existence of better estimates in negative norms is a
corner stone in the derivation of the improvement of the error estimate
(\ref{eq:error_lambda_F}) for the eigenvalues (doubling of the
convergence rate), we expect that the eigenvalue approximation will be
dramatically polluted by the use of the numerical integration
formula. 

This can be checked numerically. Considering again the
one-dimensional example used in Section~\ref{sec:Fourier}
($\Omega=(0,2\pi)$, $V(x) = \sin(|x-\pi|/2)$, $F(t)=t^2/2$), we
have computed for $4 \le N \le 30$ and $N_g=2^p$ with $7 \le p \le 15$,
the errors $\|u_{N,N_g}-u\|_{H^1}$, $\|u_{N,N_g}-u\|_{L^2}$,
$\|u_{N,N_g}-u\|_{H^{-1}}$, and $|\lambda_{N,N_g}-\lambda|$. 
On Figure~3, these quantities are plotted as functions
of $2N+1$ (the dimension of $\widetilde X_N$), for various values of
$N_g$. 

The non-monotonicity of the curve $N \mapsto
  |\lambda_{N,N_g}-\lambda|$  originates from the fact that
  $\lambda_{N,N_g}-\lambda$ can be positive or negative depending on the
  values of $N$ and $N_g$. 

\medskip

\begin{figure}[h] \label{fig:FourierNg}
$\!\!\!\!\!\!\!\!\!\!\!\!\!\!\!\!\!\!\!$
\begin{tabular}{cc}
\psfig{figure=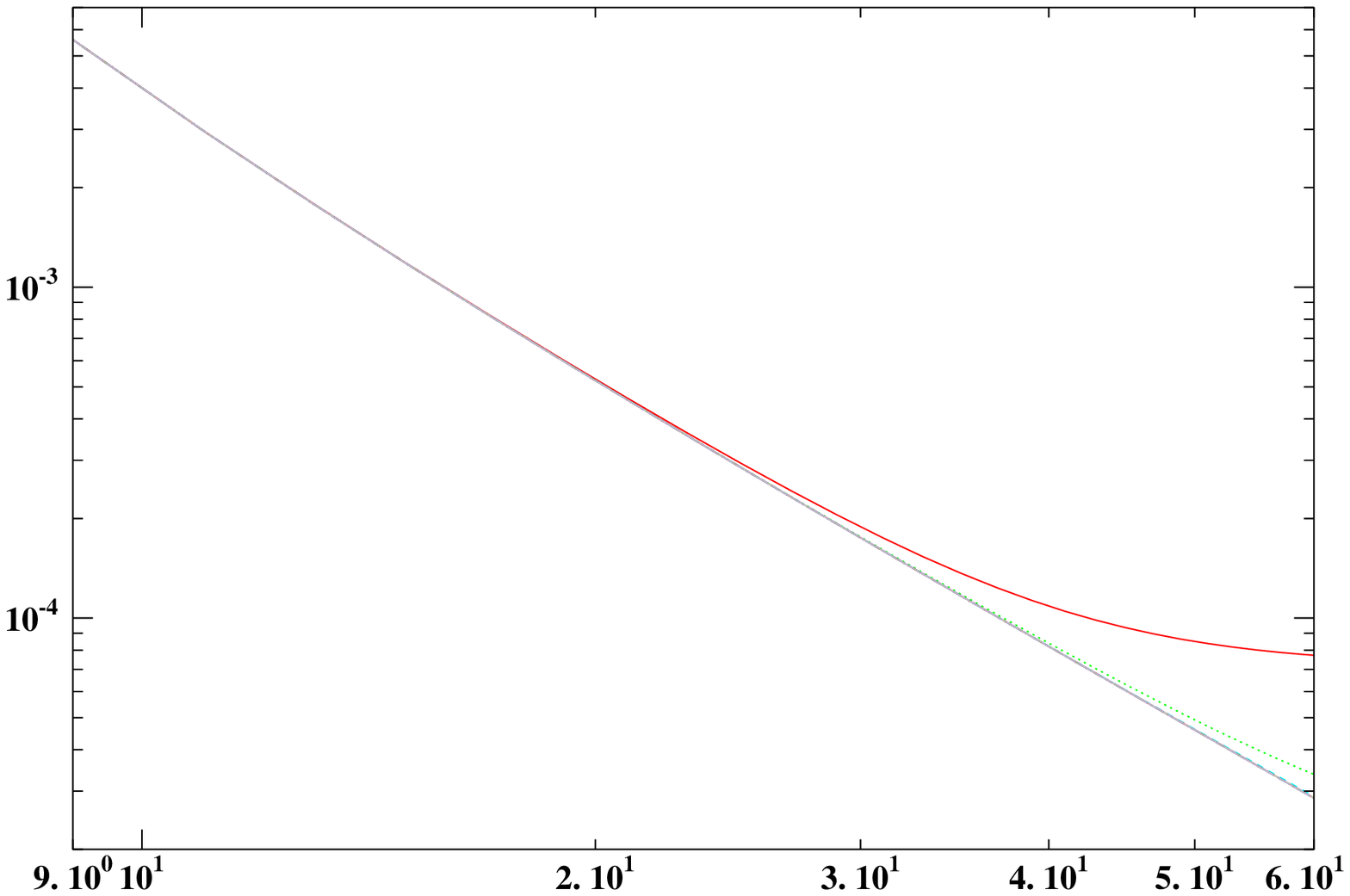,height=4.5truecm} &
\psfig{figure=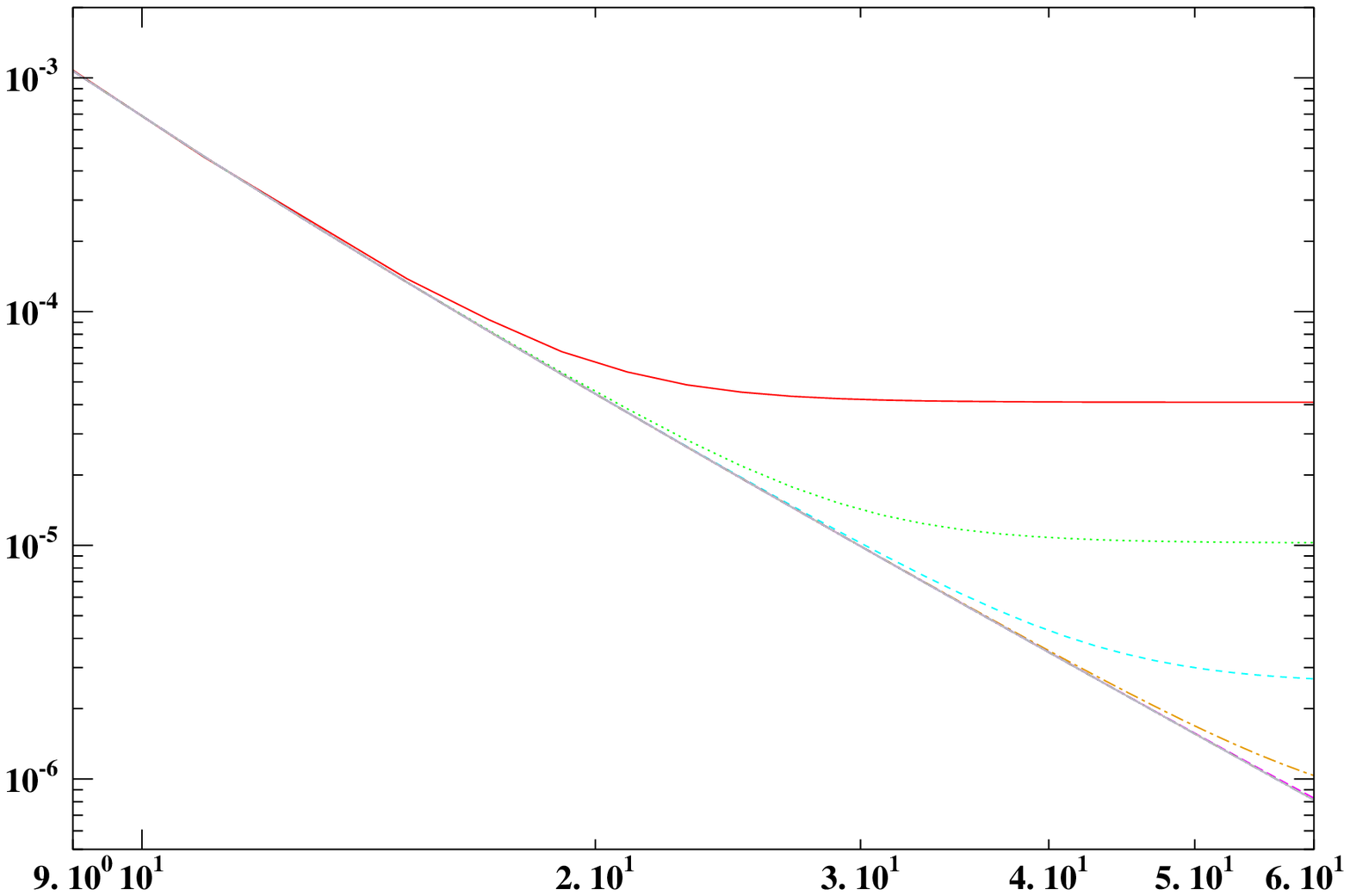,height=4.5truecm} \\
\psfig{figure=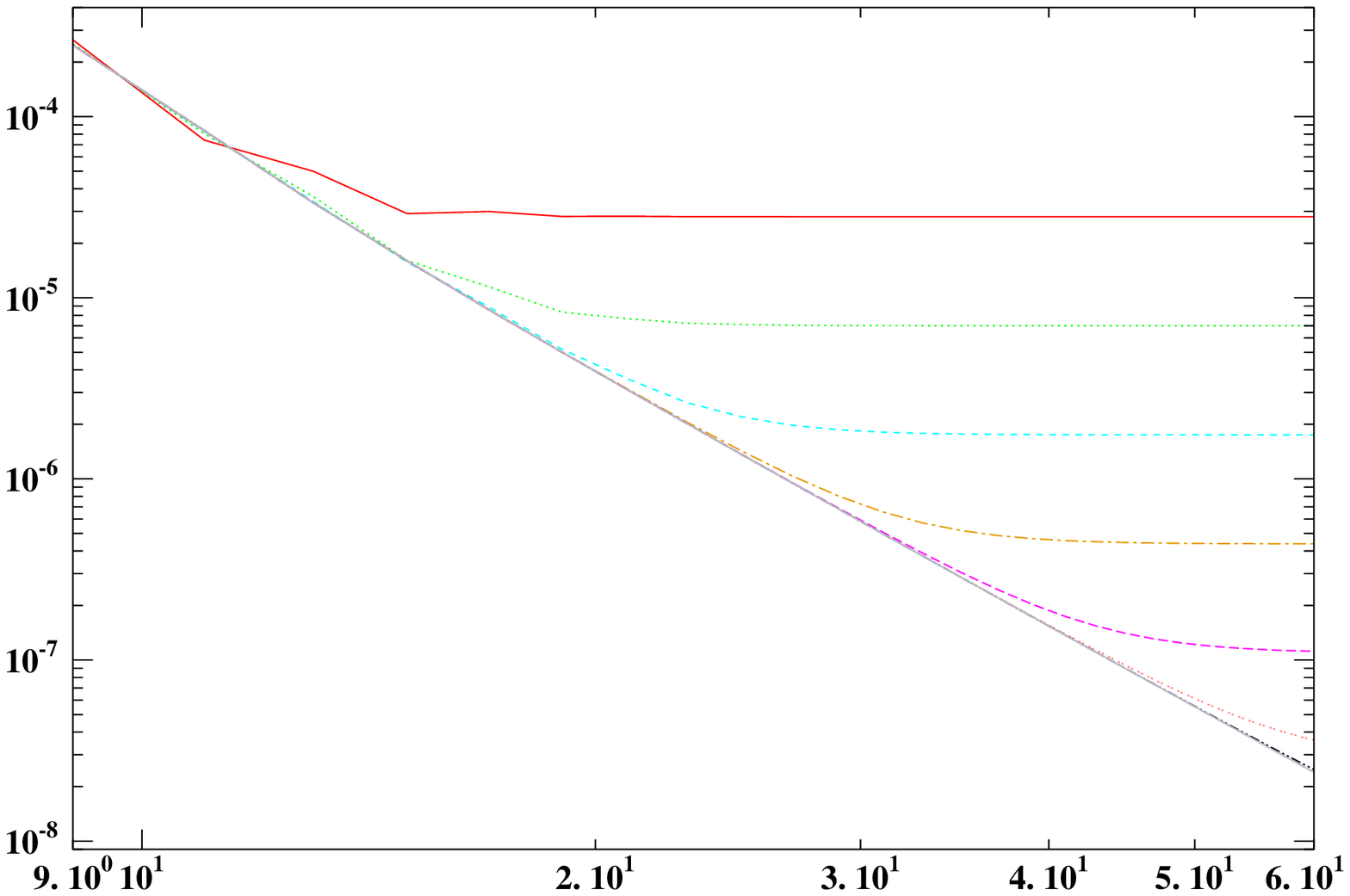,height=4.5truecm} &
\psfig{figure=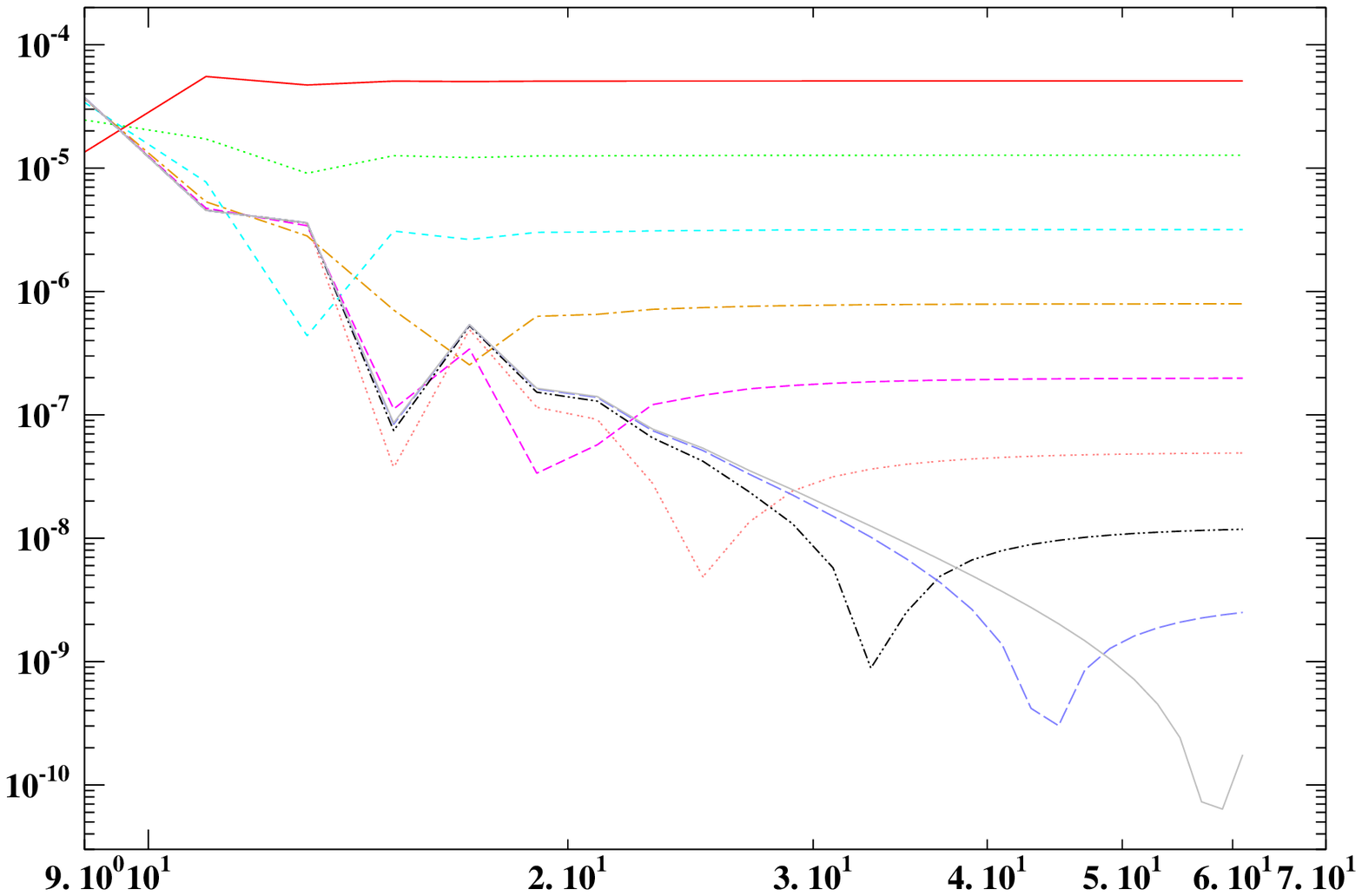,height=4.5truecm} 
\end{tabular}
\caption{Numerical errors $\|u_{N,N_g}-u\|_{H^1}$ (top left),
  $\|u_{N,N_g}-u\|_{L^2}$ (top right), $\|u_{N,N_g}-u\|_{H^{-1}}$
  (bottom left), and $|\lambda_{N,N_g}-\lambda|$ (bottom right), as
  functions of $2N+1$ (the dimension of $\widetilde X_N$), for $N_g=128$ (red),
  $N_g=256$ (green), $N_g=512$ (cyan), $N_g=1024$ (gold), $N_g=2048$
  (magenta), $N_g=4096$ (pink), $N_g=8192$ (black), $N_g=16384$ (blue),
  $N_g=32768$ (light blue).}
\end{figure}

\medskip

The numerical errors $\|u_{N,N_g}-u\|_{H^1}$,
$\|u_{N,N_g}-u\|_{L^2}$, $\|u_{N,N_g}-u\|_{H^{-1}}$, and
$|\lambda_{N,N_g}-\lambda|$, for $N=30$, as functions of $N_g$ (in log
scales) are plotted on Figure~4. When $N_g$ goes to 
infinity, the sequences $\log_{10} \|u_{N,N_g}-u\|_{H^1}$,
$\log_{10} \|u_{N,N_g}-u\|_{L^2}$, $\log_{10} \|u_{N,N_g}-u\|_{H^{-1}}$, and
$\log_{10} |\lambda_{N,N_g}-\lambda|$ converge to $\log_{10}
\|u_{N}-u\|_{H^1}$, $\log_{10} \|u_{N}-u\|_{L^2}$, $\log_{10}
\|u_{N}-u\|_{H^{-1}}$, and $\log_{10} |\lambda_{N}-\lambda|$
respectively. For smaller values of $N_g$, the numerical integration
error dominates and these functions all decay
linearly with $\log_{10}N_g$ with a slope very close to $-2$. For fixed
$N$, the upper bounds (\ref{eq:erroruNNgL2})-(\ref{eq:errorlambdaNNg}) 
also decay linearly with $\log_{10}N_g$, but with a slope equal to
$-1.5$. To obtain sharper upper bounds for the numerical integration
error, we need to replace 
(\ref{eq:estimeIV}) with a sharper estimate of 
$\|\Pi_{2N}(V-{\cal I}_{N_g}(V))\|_{L^2}$, which is possible for the
particular example under consideration here. Indeed, remarking that
under the condition $N_g \ge 4N+1$, 
$$
\|\Pi_{2N}(V-{\cal I}_{N_g}(V))\|_{L^2}
= \left( \sum_{|g| \le 2N} \left| 
\sum_{k \in \Z^\ast} \widehat V_{g+kN_g} \right|^2 \right)^{1/2},
$$
we can, using (\ref{eq:coeff_Fourier}), show that
$$
\|\Pi_{2N}(V-{\cal I}_{N_g}(V))\|_{L^2} \le \frac{C \, N^{1/2}}{N_g^2},
$$
for a constant $C$ independent of $N$ and $N_g$. We deduce that for this
specific example
\begin{eqnarray*} 
\|u_{N,N_g}-u\|_{H^1} & \le & C \left(N^{-5/2}+N^{1/2} N_g^{-2}\right) \\
\|u_{N,N_g}-u\|_{L^2} & \le & C \left(N^{-7/2}+N^{1/2} N_g^{-2}\right)
\\
|\lambda_{N,N_g}-\lambda| & \le & C
\left(N^{-9/2}+N^{1/2} N_g^{-2}\right).  
\end{eqnarray*}

\medskip

\begin{figure}[h] \label{fig:FourierNg2}
\centering
\psfig{figure=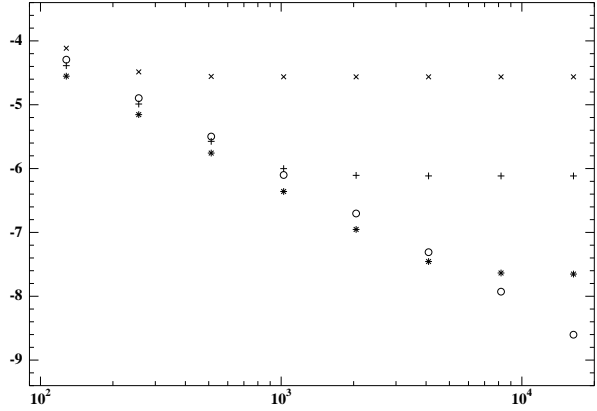,height=6truecm}
\caption{Numerical errors $\|u_{N,N_g}-u\|_{H^1}$ ($\times$),
$\|u_{N,N_g}-u\|_{L^2}$ ($+$), $\|u_{N,N_g}-u\|_{H^{-1}}$ ($\ast$), and
$|\lambda_{N,N_g}-\lambda|$ ($\circ$), for $N=30$, as functions of $N_g$
(in log scales).}
\end{figure}

\section*{Acknowledgements} This work was done while E.C.
was visiting the Division of Applied Mathematics of Brown
University, whose support is gratefully acknowledged. The authors also
thank Jean-Yves Chemin and Didier Smets for fruitful discussions, and
Claude Le Bris for valuable comments on a preliminary version of this work.

\section{Appendix: properties of the ground state}

The mathematical properties of the minimization problems
(\ref{eq:min_pb_u}) and (\ref{eq:min_pb_rho}) which are useful for the
numerical analysis reported in this article are gathered in the
following lemma.  

\medskip

Recall that $d=1$, $2$ or $3$.

\medskip

\begin{lem} \label{lem:theory}
Under assumptions (\ref{eq:Hyp1})-(\ref{eq:Hyp7}), 
(\ref{eq:min_pb_rho}) has a unique minimizer $\rho_0$ and
(\ref{eq:min_pb_u}) has exactly two minimizers $u=\sqrt{\rho_0}$ and
$-u$. The function $u$ is solution to the nonlinear eigenvalue problem 
(\ref{eq:10}) for some $\lambda \in \R$. 
Besides, $u \in C^{0,\alpha}(\overline{\Omega})$ for some $0 <
\alpha < 1$, $u > 0$ in $\Omega$, and $\lambda$ is the lowest eigenvalue
of $A_u$ and is non-degenerate.   
\end{lem}

\medskip

\begin{proof}
As $A$ is uniformly bounded and coercive on $\Omega$ and $V \in
L^q(\Omega)$ for some $q > \max(1,d/2)$, $v \mapsto a(v,v)$ is a
quadratic form on $X$, bounded from
below on the set $\left\{v \in X \; | \; \|v\|_{L^2}=1\right\}$.
Replacing $a(v,v)$ with $a(v,v)+C\|v\|_{L^2}^2$ and $F(t)$ with
$F(t)-F(0)-tF'(0)$ 
does not change the minimizers of (\ref{eq:min_pb_u}) and
(\ref{eq:min_pb_rho}). We can therefore assume, {\em without loss of
generality}, that 
\begin{equation} \label{eq:hyp9}
\forall v \in X, \; a(v,v) \ge \|v\|_{L^2}^2 \quad \mbox{ and } \quad 
F(0)=F'(0)=0. 
\end{equation}
It then follows from (\ref{eq:Hyp7}) and (\ref{eq:hyp9}) that
$0 \le F(v^2) \le C (v^2+v^{6})$. As $X 
\hookrightarrow L^{6}(\Omega)$, $E(v)$ is finite for
all $v \in X$, $I > -\infty$ and the minimizing sequences of
(\ref{eq:min_pb_u}) are bounded in $X$. Let $(v_n)_{n \in \N}$ be a
minimizing sequence of (\ref{eq:min_pb_u}). Using the fact that $X$ is
compactly embedded in $L^2(\Omega)$, we can extract from $(v_n)_{n \in \N}$ a
subsequence $(v_{n_k})_{k \in \N}$ which converges weakly in $X$,
strongly in $L^2(\Omega)$ and almost everywhere in $\Omega$ to some $u \in
X$. As $\|v_{n_k}\|_{L^2}=1$ and $E(v_{n_k}) \downarrow I$, we obtain
$\|u\|_{L^2}=1$ and $E(u) \le I$ ($E$ 
is convex and strongly continuous, hence weakly l.s.c., on $X$). Hence
$u$ is a minimizer of (\ref{eq:min_pb_u}). As $|u| \in X$,
$\||u|\|_{L^2} = 1$ and $E(|u|)=E(u)$, we can assume without loss of
generality that $u \ge 0$. Assumptions (\ref{eq:Hyp1})-(\ref{eq:Hyp7})
imply that $E$ is $C^1$ on $X$ and that $E'(u) = A_u u$. It follows that
$u$ is solution to (\ref{eq:Euler}) for some $\lambda \in \R$. By
elliptic regularity arguments~\cite{GT}, we get $u \in
C^{0,\alpha}(\overline{\Omega})$ for some
$0 < \alpha < 1$. We also have $u >
0$ in $\Omega$; this is a consequence of the Harnack 
inequality~\cite{Stampacchia}. Making the change of variable 
$\rho=v^2$, it is easily seen that if $v$ is a minimizer of
(\ref{eq:min_pb_u}), then $v^2$ is a minimizer of (\ref{eq:min_pb_rho}),
and that, conversely, if $\rho$ is a minimizer of (\ref{eq:min_pb_rho}),
then $\sqrt{\rho}$ and $-\sqrt{\rho}$ are minimizers of
(\ref{eq:min_pb_u}). Besides, the functional $\cal E$ is strictly
convex on the convex set $\left\{ \rho \ge 0 \; | \, \sqrt{\rho} \in X,
  \; \int_\Omega \rho = 1 \right\}$. Therefore $\rho_0=u^2$ is the unique
minimizer of (\ref{eq:min_pb_rho}) and $u$ and $-u$ are the only
minimizers of (\ref{eq:min_pb_u}).

It is easy to see that $A_u$ is bounded below and has a compact
resolvent. It therefore possesses a lowest eigenvalue $\lambda_0$,
which, according to the min-max principle, satisfies
\begin{equation} \label{eq:minmax}
\lambda_0 = \inf \left\{ \int_{\Omega} (A \nabla v) \cdot \nabla v +
  \int_\Omega   (V+f(u^2)) v^2, \; v \in X, \; \int_\Omega v^2=1 \right\}.
\end{equation}
Let $v_0$ be a normalized eigenvector of $A_u$ associated with
$\lambda_0$. Clearly, $v_0$ is a minimizer of (\ref{eq:minmax}) and
so is $|v_0|$. Therefore, $|v_0|$ is solution to the Euler equation
$A_u|v_0|=\lambda_0|v_0|$. Using again elliptic regularity arguments and
the Harnack inequality, we obtain that $|v_0| \in
C^{0,\alpha}(\overline{\Omega})$ for some
$0 < \alpha < 1$ and that $|v_0| > 0$ on $\Omega$. This implies that either
$v_0 = |v_0| > 0$ in $\Omega$ or $v_0=-|v_0| < 0$ in $\Omega$. In
particular $(u,v_0)_{L^2} \neq 0$. Consequently, $\lambda=\lambda_0$
and $\lambda$ is a simple eigenvalue of $A_u$. 
\end{proof}

\medskip

Let us finally prove that
$\lambda$ is also the ground state eigenvalue of the {\em nonlinear}
eigenvalue problem  
\begin{equation} \label{eq:nonlinear_eigenvalue_pb}
\left\{ \begin{array}{l}
\mbox{search } (\mu,v) \in \R \times X \mbox{ such that} \\
A_v v = \mu v \\
\| v \|_{L^2} = 1, \end{array} \right. 
\end{equation}
in the following sense: if $(\mu,v)$ is solution to
(\ref{eq:nonlinear_eigenvalue_pb}) then either $\mu > \lambda$ or
$\mu=\lambda$ and $v= \pm u$. 

To see this, let us consider a solution $(\mu,v) \in \R \times X$ to
(\ref{eq:nonlinear_eigenvalue_pb}) and denote by $\widetilde w = |v|-u$. 
As for $u$, we infer from elliptic regularity arguments~\cite{GT} that $v \in
C^{0,\alpha}(\overline{\Omega})$.
We have $\|v\|_{L^2} = \|u\|_{L^2} = 1$. Therefore, if $w \le 0$ in
$\Omega$, then 
$|v|=u$, which yields $v=\pm u$ and $\mu=\lambda$. Otherwise, there exists
$x_0 \in \Omega$ such that $\widetilde w(x_0) > 0$, and, up to replacing
$v$ with $-v$, we can consider that the function $w=v-u$ is such that
$w(x_0) > 0$. The function $w$ is in $X \cap
C^{0,\alpha}(\overline{\Omega})$ and satisfies
\begin{equation} \label{eq:on_w}
(A_u-\lambda)w + \frac{f(v^2)-f(u^2)}{v^2-u^2} v (u+v) w = (\mu-\lambda) v.
\end{equation} 
Let $\omega = \left\{ x \in \Omega \; | \; w(x) > 0
\right\} = \left\{ x \in \Omega \; | \; v(x) > u(x) \right\}$ and $w_+ =
\max(w,0)$. As $w_+ \in X$, we deduce from (\ref{eq:on_w}) that
$$
\langle (A_u-\lambda)w_+,w_+\rangle_{X',X} + \int_\omega 
\frac{f(v^2)-f(u^2)}{v^2-u^2} v (u+v) w^2 = (\mu-\lambda) \int_\omega vw.
$$
The left hand side of the above equality is positive and $\int_\omega vw
> 0$. Therefore, $\mu > \lambda$.

\bibliographystyle{plain}
\bibliography{article}

\end{document}